\documentclass[11pt]{amsart}
\usepackage{amsmath,amssymb,amsthm,a4wide,color}
\usepackage{comment}
\usepackage{graphicx, color}
\usepackage{comment}

\newtheorem{theorem}{Theorem}[section]    
\newtheorem{lemma}[theorem]{Lemma}          
\newtheorem{proposition}[theorem]{Proposition}  
\newtheorem{claim}[theorem]{Claim}  

\newtheorem{corollary}[theorem]{Corollary} 
\newtheorem{conjecture}[theorem]{Conjecture} 

\theoremstyle{definition}
\newtheorem{definition}[theorem]{Definition}

\newtheorem{remark-number}[theorem]{Remark}  
\newtheorem{question}[theorem]{Question} 
 
\newtheorem{example}[theorem]{Example}   
\newtheorem{remark}{Remark}             
\numberwithin{equation}{section}

\newcommand{\K}{\mathcal{K}}
\newcommand{\T}{\mathcal{T}}
\newcommand{\Z}{\mathbb{Z}}
\newcommand{\R}{\mathbb{R}}

\title{
Positivities of knots and links and the defect of Bennequin inequality
}

\author{Jesse Hamer}
\address{Department of Mathematics,   
The University of Iowa, Iowa City, IA 52242, USA}
\email{jesse-hamer@uiowa.edu} 
\author{Tetsuya Ito}
\address{Department of Mathematics, Kyoto University, Kyoto 606-8502, JAPAN}
\email{tetitoh@math.kyoto-u.ac.jp}
\author{Keiko Kawamuro}
\address{Department of Mathematics,   
The University of Iowa, Iowa City, IA 52242, USA}
\email{keiko-kawamuro@uiowa.edu}

\date{\today}

\begin{document}

\maketitle
\begin{abstract}
We discuss relations among various positivities of knots and links, such as strong quasipositivity and quasipositivity. We give several pieces of supporting evidence for conjectural statements concerning these positivities and the defect of Bennequin inequality. Finally, we determine strong quasipositivity and quasipositivity for knots up to 12 crossings (with two exceptions for quasipositivity).
\end{abstract}

\tableofcontents

\section{Introduction}

Various positivities of knots, links and braids play important roles in knot theory and contact geometry. 
They are defined diagramatically.  It is a fundamental problem to determine when a given link (or braid) has these positivities.

The aim of this paper is to discuss relations among the various notions of positivity and give several conjectural implications. We give supporting evidence for the conjectures 
by proving them under additional assumptions and by determining  
\begin{itemize}
\item
all the strongly quasipositive knots up to 12 crossings and 
\item 
all the quasipositive knots with two exceptions up to 12 crossings. 
\end{itemize} 
Moreover, for these strongly quasipositive and quasipositive knots we give explicit strongly quasipositive braid words and quasipositive braid words respectively in Section \ref{sec:tables}.

\subsection{Notation and conventions}
In this section we list our notation and conventions, and review the definitions of various positivities of braids and links.

Let $B_n$ be the braid group presented by
$$ 
B_n= \left\langle \sigma_1,\dots,\sigma_{n-1} \middle| \begin{array}{ll} \sigma_i \sigma_j = \sigma_i \sigma_j & ( |i-j| >1)  \\
\sigma_i \sigma_{i+1} \sigma_i = \sigma_{i+1} \sigma_i \sigma_{i+1} 
& (i=1,\dots, n-2)
\end{array}
\right\rangle
$$ 
We call the generators $\sigma_1,\ldots, \sigma_{n-1}$ the \emph{standard} generators of $B_n$. 
For $1\leq i<j\leq n$, the \emph{band generator} $\sigma_{i,j}$ is defined by
\[\sigma_{i,j}=(\sigma_{j-1}\sigma_{j-2}\cdots \sigma_{i+1})\sigma_{i}(\sigma_{j-1}\sigma_{j-2}\cdots \sigma_{i+1})^{-1}. \]
Using the band generators, the braid group is presented by 
\[ B_n= \left\langle \sigma_{i,j} \ \middle| \begin{array}{ll}
\sigma_{j,k}\sigma_{i,j} = \sigma_{i,j}\sigma_{i,k}=\sigma_{i,k}\sigma_{j,k}
& (i<j<k)\\
\sigma_{i,j}\sigma_{k,l}=\sigma_{k,l}\sigma_{i,j} & (i<j<k<l)
\end{array}
 \right\rangle\]

\begin{definition}[Positivities of braids]
A braid $\beta \in B_{n}$ is
\begin{itemize}
\item[--] \emph{positive} if it is written as a product of positive powers of some of the standard generators  $\sigma_1,\ldots, \sigma_{n-1}$. 
\item[--] \emph{strongly quasipositive} if it is written as a product of positive powers of some of the band generators $\{\sigma_{i,j} \ | \ 1\leq i < j \leq n\}$.
\item[--] \emph{quasipositive}  if it is written as a product of positive powers of 
some conjugates of the standard generators $\sigma_1,\sigma_2,\dotsc, \sigma_{n-1}$.
\end{itemize}
\end{definition}

Throughout the paper, for simplicity we mean by  ``link"  either an oriented knot or link. To distinguish topological links, transverse links and their closed braid representatives, we use the following notation.
\begin{itemize}
\item The letter $\mathcal K$ denotes a topological knot or link type in $S^3$. 

\item The letter $\mathcal T$ denotes a transverse link in the standard contact 3-sphere $(S^3, \xi_{std})$. 

\item The letter $K$ denotes a braid word (in the standard generators or the band generators) whose closure represents $\K$ or $\mathcal T$.
For braid words $K$ and $K'$ we write:
\begin{itemize}
\item $K=K'$ if $K$ and $K'$ are exactly the same braid word.
\item $K \sim K'$ if $K$ and $K'$ represent the same element of the braid group $B_n$.
\item $K \approx K'$ if the braids represented by $K$ and $K'$ are conjugate in $B_n$.
\end{itemize} 
\end{itemize}

\begin{definition}[Positivities of links]
A link $\K$ is 
\begin{itemize}
\item[--] a \emph{positive braid link} if $\mathcal{K}$ can be represented by a positive braid.
\item[--] a \emph{positive link} if $\mathcal{K}$ can be represented by a diagram without negative crossings.
\item[--] a \emph{strongly quasipositive link} if $\mathcal{K}$ can be represented by a strongly quasipositive braid.

\item[--] a \emph{quasipositive link} if $\mathcal{K}$ can be represented by a quasipositive braid.

\item[--] an \emph{almost positive braid link} if $\mathcal{K}$ can be represented by a braid word $K$ in the standard generators and their inverses, $\sigma_1^{\pm 1},\ldots,\sigma_{n-1}^{\pm 1}$, such that the number of inverses used is at most one. 
(\emph{Almost strongly quasipositive link} and  \emph{almost quasipositive link} are defined similarly.) 

\item[--] an \emph{almost positive link} if $\mathcal{K}$ is represented by a diagram which has at most one negative crossing.
\end{itemize}

\end{definition}

The band generator $\sigma_{i,j}$ (resp. its inverse $\sigma_{i,j}^{-1}$) can be viewed as the boundary of a positively (resp. negatively) twisted band connecting the $i$-th and $j$-th strands of the braid. 
The following special Seifert surface plays an important role.

\begin{definition}[Bennequin surface]
Let $K$ be an $n$-braid word. 
Starting with $n$ parallel disks and attaching a twisted band for each letter $\sigma_i^{\pm 1}$ or $\sigma_{i,j}^{\pm 1}$ in the word $K$, we get a Seifert surface of the topological link type $\mathcal{K}$ which is denoted by $\Sigma_K$. We call $\Sigma_K$ the \emph{Bennequin surface} associated to the braid word $K$.
\end{definition}

We also recall the definition of a homogeneous link since later we study the above positivities for homogeneous links. 

\begin{definition}[Canonical Seifert surface]
Given a link diagram $D \subset \mathbb{R}^{2}$, Seifert's algorithm \cite{Rolfsen} 
gives a Seifert surface $F_D$ which we call the \emph{canonical Seifert surface} of the diagram $D$. 
We describe $F_D$ by the Seifert circles and signed edges that encode the signs of the corresponding crossings in $D$. See Figure~ \ref{fig:canonical_Seifert} (ii). 

The \emph{Seifert graph} $G_D$ consists of vertices corresponding to Seifert circles of $D$ and signed edges corresponding to twisted bands of $F_D$ where the sign encodes the sign of the corresponding crossing. See Figure \ref{fig:canonical_Seifert} (iii).

A vertex $v$ of $G_D$ is called a \emph{cut vertex} if $G_D\setminus v$ is disconnected. A \emph{block} is a maximal connected subgraph of $G_D$ containing no cut vertices.  
See Figure \ref{fig:canonical_Seifert}  (iv). 
\end{definition} 

\begin{figure}[htbp]
\begin{center}
\includegraphics*[width=90mm]{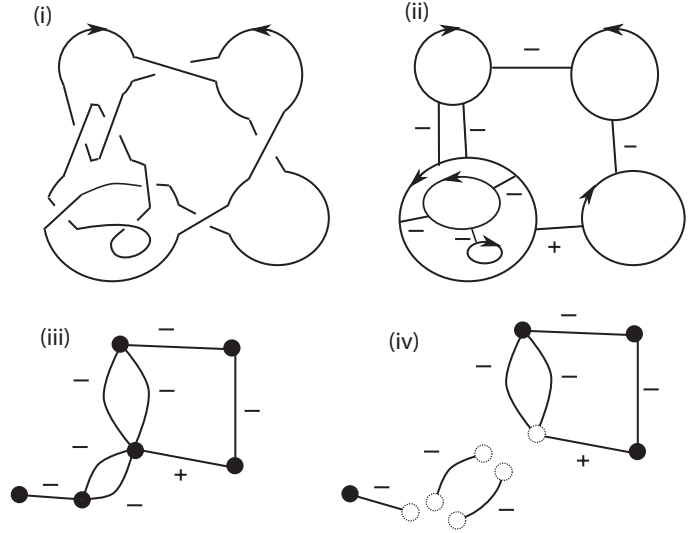}
\caption{
(i) Link diagram $D$. 
(ii) Circle-edge description of canonical Seifert surface $F_{D}$.
(iii) Seifert graph $G_{D}$. 
(iv) Blocks of $G_D$. }
\label{fig:canonical_Seifert}
\end{center}
\end{figure}

\begin{definition}[Homogeneous link] \label{def:homogeneous} 
A link diagram is called \emph{homogeneous} if for each block of $G_D$ all the edges have the same sign. A link admitting a homogeneous diagram is called  \emph{homogeneous}. 
\end{definition}

Alternating diagrams and positive diagrams are homogeneous.  
One useful property of a homogeneous diagram $D$ we use is 
\begin{equation}\label{eq:useful property}
\chi(F_D)=\chi(\K)
\end{equation}
see \cite[Corollary 4.1]{cr}. For other basic properties of homogeneous links, we refer the reader to \cite{cr}. 

The \emph{braid index} $\beta(\K)$ (resp. $\beta(\mathcal T)$) is the minimum number of strands which is needed to represent $\K$ (resp. $\mathcal{T}$) as a closed braid.

Among homogeneous links, we will often use the following more restricted subclasses.

\begin{definition}[Strongly homogeneous link]
\label{def:strongly-homogeneous}
We say that a link $\K$ is \emph{strongly homogeneous} if $\K$ admits a homogeneous diagram $D$ with $s(D)=\beta(\K)$, where $s(D)$ denotes the number of Seifert circles of $D$.
\end{definition}

We denote the \emph{self-linking number} of a transverse link $\mathcal T$ by 
$sl(\mathcal T)$. When $\mathcal{T}$ is represented by a closed $n$-braid $K$, we have 
\[ sl(\mathcal{T}) = -n + w(K) \]
where $w(K)$ denotes the writhe (exponent sum) of the braid $K$ \cite{Bennequin}.
For a topological link $\K$ the {\em maximal self-linking number}  $SL(\K)$ is defined by
\[ SL(\K) = \max\{ sl(\mathcal T) \ | \ \mathcal T \mbox{ is a transverse link representative of } \mathcal K\}.\]
Thanks to the generalized Jones' conjecture \cite{K} proven in \cite{DynnikovPrasolov, LM}, we have 
\[ SL(\K)=-\beta(\K) + w(\K)\]
where $w(\K)$ is the uniquely determined writhe of a braid representative of $\K$ realizing the braid index $\beta(\K)$.

\subsection{The defect $\delta_3$} 

Let
\begin{align*} 
\chi(\K) &= \max\{ \chi(\Sigma) \ | \ \Sigma \mbox{ is a Seifert surface for }  \K\}
\end{align*}
be the maximal Euler characteristic for $\K$.
It was discovered by Bennequin \cite{Bennequin} that 
\begin{equation}\label{eq:Benne-ineq}
SL(\K) \leq -\chi(\K),
\end{equation}
which we call the {\em Bennequin inequality.}
A number of people, including Etnyre, Hedden, Rudolph, and Van Horn-Morris, have been studying the relation between the sharpness of the Bennequin inequality and strong quasipositivity of $\K$.

Let $\K$ be a link type in $S^3$ and $K$ be a braid representative of $\K$. 
By Bennequin \cite{Bennequin} we may regard $K$ as a transverse link, $\T$, in $(S^3, \xi_{std})$. In \cite[Definition 1.1]{IK-bennequin-bound} we introduced 
$$\delta_3(\T)=\delta_3(K):= \frac{1}{2}(-\chi(\K) - sl(K)),$$ 
which we call the {\em defect of the transverse link type of $K$}.   
This leads us to further define $\delta_3(\K)$ as follows: 

\begin{definition}
Define 
\begin{eqnarray*}
\delta_3(\K)&:=& \frac{1}{2}(-\chi(\K) - SL(\K)) \\
&=& \min\{\delta_3(K) \ | \ K \mbox{ is a braid representative of } \K\}.
\end{eqnarray*}
We call $\delta_3(\K)$ the \textit{defect} of  the topological link type $\K$. 
\end{definition}

Since $\chi(\K)\geq \chi(\Sigma_K)= n - \mbox{(the length of the braid word } K)$
and the band generator $\sigma_{i,j}^{\pm 1}$ contributes $\pm 1$ to the writhe $w(K)$, we obtain 
$$\delta_3(K)\leq  \mbox{ the number of negative bands in }\Sigma_K.$$
In particular, $\delta_3(K)=0$ if $K$ is strongly quasipositive.

The geometric meaning of the defect $\delta_3$ is expected to be the number of negative bands in a minimal genus Bennequin surface:

\begin{conjecture}\label{SQPconjecture}\cite[Conjecture 1]{IK-bennequin-bound}
Let $\T$ be a transverse link of topological type $\K$. 
There exists a braid representative $K$ of $\T$ such that 
\[ \delta_3(\T)=
\mbox{ the number of negative bands in } \Sigma_K;\] 
equivalently, 
\[\chi(\Sigma_K)=\chi(\K).\]
Consequently, there exists a braid representative $K'$ of $\K$ such that 
\[
\delta_3(\K)=\mbox{ the number of negative bands in } \Sigma_{K'}; 
\]
equivalently,  \[\chi(\Sigma_{K'})=\chi(\K) \mbox{ and } sl(K')=SL(\K).
\]
\end{conjecture}

\subsection{The defect $\delta_4$} 

Let
\begin{align*} 
\chi_4(\K) &= \max\{ \chi(\Sigma) \ | \ \Sigma \subset B^4 \mbox{ is an oriented smoothly embedded surface with } \partial\Sigma= \K \}.
\end{align*}
The 4-dimensional counterpart of the Bennequin inequality (\ref{eq:Benne-ineq}), called the {\em slice Bennequin inequality}, was proved by Lisca and Mati\'c \cite{LM} and, independently, Akbulut and Matveyev \cite{AM} (see also Rudolph \cite{R2}): 
\begin{equation} 
SL(\K) \leq - \chi_4(\K).
\end{equation}
As a 4-dimensional counterpart of $\delta_3$ we introduce the following.
\begin{definition}
Define 
$$\delta_4(\K)= \frac{1}{2} (-\chi_4(\K)-SL(\K))$$
and call it the {\em defect of the slice Bennequin inequality} for the topological link type $\K$.
\end{definition}

Here is a basic well-known fact (see for example \cite{EV}):
\begin{proposition}\label{prop:delta4=0}
Every quasipositive link $\K$ has $\delta_4(\K)=0$. 
\end{proposition}

Using the contrapositive of Proposition~\ref{prop:delta4=0} we can detect non-quasipositive links: 

\begin{corollary}
The following ten knots, whose non-quasipositivity was previously unknown according to KnotInfo are non-quasipositive. 
\[ 11_{n37},\ 12_{n120},\ 12_{n199},\ 12_{n200},\ 12_{n260},\ 12_{n312},\ 12_{n397},\ 12_{n414},\ 12n_{523},\ 12n_{549}\] 
\end{corollary}
 
The converse of  Proposition \ref{prop:delta4=0} has been asked by a number of people (see \cite[Question 7.3]{EV}, for example). 

\begin{question}\label{question:d4=0}
Does $\delta_4(\K)=0$ imply that $\K$ is quasipositive?
\end{question}

In Section \ref{sec:tables}, we give a list of knots with $\delta_4(\K)=0$ up to 12 crossings. Here for the knots $12_{n239}$ and $12_{n512}$, $\delta_{4}(\K)$ is only known to be either $0$ or $1$, and we could not determine they are quasipositive or not. However, with this two exceptions, we confirmed that all of knots with $\delta_4(\K)=0$ up to 12 crossings are indeed quasipositive. Consequently, the answer to the question is ``Yes'' for knots up to 12 crossings (possibly two exceptions). 

Also, we may ask a generalization of this question: 

\begin{question}\label{question:d4=1}
Does $\delta_4(\K)=1$ imply that $\K$ is almost quasipositive?
\end{question}

\subsection{Relations among positivities and summary of our results}

In Figure \ref{fig:chart}, we summarize several (conjectural) relations among various positivities and results of this paper.  
Dotted arrows represent conjectural implications.

Shortly before publicizing this paper, the authors learned that in \cite{FLL}, Feller, Lewark, and Lobb prove that almost positive links are strongly quasipositive.

\begin{figure}[htbp]
\begin{center}
\includegraphics*[width=150mm, bb=148 546 463 713]{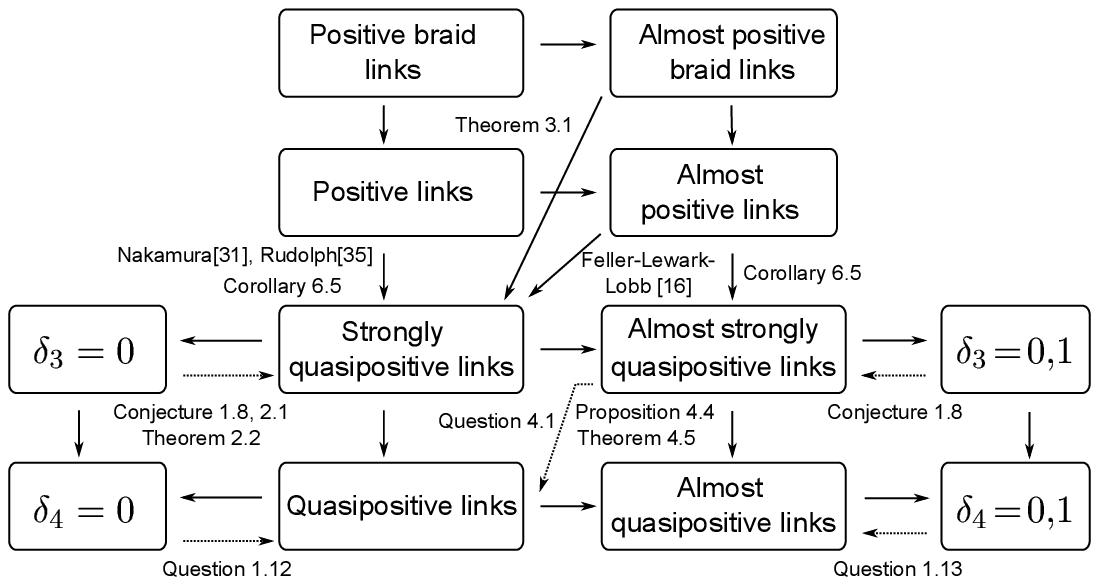}
\caption{(Conjectural) relations among various positivities. The implication (*) is  
 }
\label{fig:chart}
\end{center}
\end{figure}

We will also study these positivities for alternating or, more generally, homogeneous links. 
When $\K$ is homogeneous, it is often conjectured that a weaker positivity implies a stronger one as summarized in Figure \ref{fig:chart2}. In particular, we expect that all the properties listed in Figure \ref{fig:chart2} are equivalent.
\begin{figure}[htbp]
\begin{center}
\includegraphics*[width=150mm, bb=148 570 463 713]{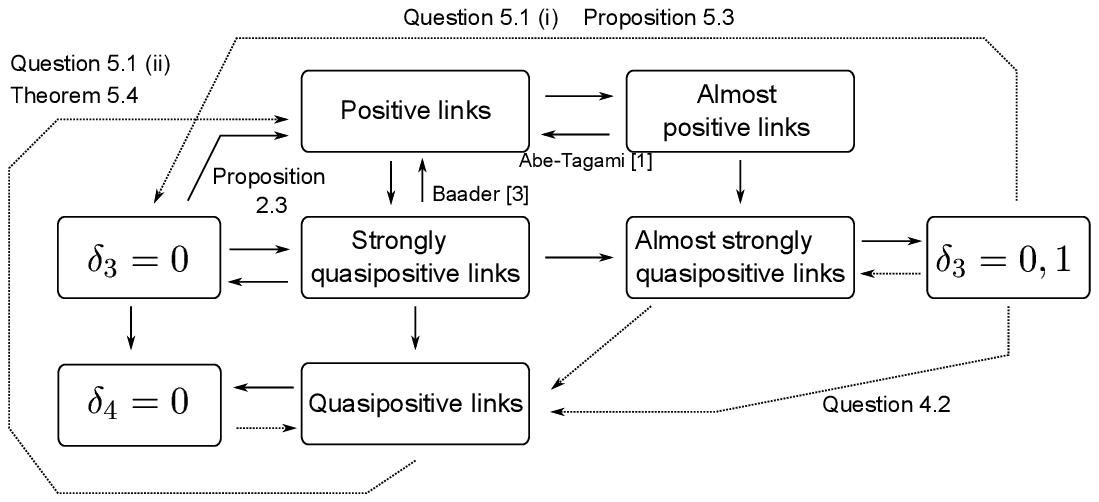}
\caption{(Conjectural) relations among various positivities for alternating (or homogeneous) links.}
\label{fig:chart2}
\end{center}
\end{figure}

\begin{remark}
The knot data we used from KnotInfo was derived from the November 10, 2016 version. Currently, KnotInfo is up to date with the quasipositive/strongly quasipositive data resulting from our work and documented in the tables of Section \ref{sec:tables}.
\end{remark}

\section{Links with $\delta_3=0$}

The following conjecture on links with $\delta_3(\K)=0$ can be found in \cite{IK-bennequin-bound} as ``Stronger Form of Conjecture 2''. 

\begin{conjecture}\label{conj:SQP}
We have $\delta_3(\mathcal K)=0$ $($equivalently $SL(\K) = -\chi(\K))$ if and only if there exists a braid representative $K$ of $\K$ such that $K$ realizes the braid index $\beta(\K)$ of $\K$ and $K$ is strongly quasipositive. 
\end{conjecture}

As for 3-braid links, Conjecture~\ref{conj:SQP} holds thanks to Birman and Manasco \cite[Theorem 1]{BM}: 

\begin{theorem}\cite[Corollary~1.10]{IK-bennequin-bound}
Let $\K$ be a link type of braid index $\beta(\K)=3$ and $K$ a $3$-braid representative of $\K$. The following are equivalent.
\begin{itemize}
\item The Bennequin inequality (\ref{eq:Benne-ineq}) is sharp for $K$;
\item $K$ is braid isotopic to a strongly quasipositive braid;
\item $\K$ is strongly quasipositive. 
\end{itemize}
\end{theorem}

A weak version of Conjecture \ref{conj:SQP} that does not take the braid index condition into account holds for homogeneous links (Definition~\ref{def:homogeneous}). 
Moreover, when $\K$ is strongly homogeneous (Definition \ref{def:strongly-homogeneous}),  Conjecture \ref{conj:SQP} holds.

\begin{proposition}\label{prop:TFAE}
Suppose that $\K$ is a homogeneous link. 
Then the following are equivalent:
\begin{enumerate}
\item[(i)] $SL(\K)=-\chi(\K)$ (equivalently $\delta_{3}(\K)=0).$
\item[(ii)] $\K$ is a positive link. 
\item[(iii)] $\K$ is strongly quasipositive. 
\end{enumerate}
Moreover, if $\K$ is strongly homogeneous then the statements {\em(i), (ii), (iii)} and 
\begin{enumerate}
\item[(iv)] $\K$ is represented by a strongly quasipositive $\beta(\K)$-braid. 
\end{enumerate}
are equivalent.
\end{proposition}

Proposition \ref{prop:TFAE} is proved in Section 6.

Here we give more supporting evidence for Conjecture \ref{conj:SQP}: 
Let $\K$ be a knot of crossing number less than or equal to 12 satisfying the equality $SL(\K) = -\chi(\K)$. 
Let $K$ denote a braid representative of $\K$ realizing $\beta(K)=\beta(\K)$. 
Such a braid word $K$ is listed in KnotInfo \cite{KnotInfo} and written in the standard generators. 
Let $n$ be the total number of negative letters appearing in the braid word $K$.
Therefore, the Bennequin surface $\Sigma_K$ associated to $K$ contains $n$ negative bands.  
We observe that there exists a sequence of Bennequin surfaces $F_0=\Sigma_K, F_1, \dots, F_m$ where $m\geq n$ such that $F_m$ is a strongly quasipositive Bennequin surface for $\K$, and 
the surface $F_{i+1}$ is obtained from $F_i$ either by 
(i) a flype,
(ii) sliding bands, or
(iii) compressing a disk that is attached to a negative band in $F_i$ so that $F_{i+1}$ has one less negative band and one less positive band.

Here, a \emph{flype} is an operation which changes a braid of the form
$\sigma_{1}^{\varepsilon}v \sigma_{1}^{m}w$ into $\sigma_{1}^{m}v\sigma_{1}^{\varepsilon}w$, where $m \in \Z$, $\varepsilon \in \{\pm 1\}$ and $w,v$ are words in $\{\sigma_{2}^{\pm 1},\ldots,\sigma_{n-1}^{\pm 1}\}$. 
(An exchange move is regarded as a special type of flype move where $m=-\varepsilon$.)  
A flype is called \emph{positive} if $\varepsilon = +1$, and called  \emph{negative} if $\varepsilon = -1$. 
A negative flype in general changes the transverse link type but does preserve the number of negative bands and positive bands in the Bennequin surface.

\begin{example}
Let $K=\sigma_1\ \sigma_1\ \sigma_2\ \sigma_1^{-1}\ \sigma_2\ \sigma_3\ \sigma_2^{-1}\ \sigma_4\ \sigma_3\ \sigma_3\ \sigma_4\ \sigma_5^{-1}\ \sigma_4\ \sigma_3^{-1}\ \sigma_4\ \sigma_5\ \sigma_5$ that represents the knot $12_{a 974}.$ 
Apply two negative flypes and obtain
$$\sigma_1^{-1}\  \sigma_2\ \sigma_1\ \sigma_1\ \sigma_2\ \sigma_3\ \sigma_2^{-1}\ \sigma_4\ \sigma_3\ \sigma_3\ \sigma_4\ \sigma_5\ \sigma_5\ \sigma_4\ \sigma_3^{-1}\ \sigma_4\ \sigma_5^{-1}.$$ 
Figure \ref{fig:flype-example} gives the progression from this negatively-flyped braid word to the following strongly quasipositive braid word (in band generators):
\[\sigma_{1,2}\sigma_{2,3}\sigma_{3,4}\sigma_{1,6}\sigma_{1,4}\sigma_{4,5}\sigma_{1,5}\sigma_{5,6}\sigma_{4,6}.\] 
This progression is characteristic of how many of the strongly quasipositive (and quasipositive) braid words appearing in our tables in Section \ref{sec:tables} were obtained. Each picture is obtained from the next via a series of slides (indicated by arrows), and compressions (indicated by shaded regions).
\end{example}

\begin{figure}[htbp]
\centering
\includegraphics[width=0.65\textwidth]{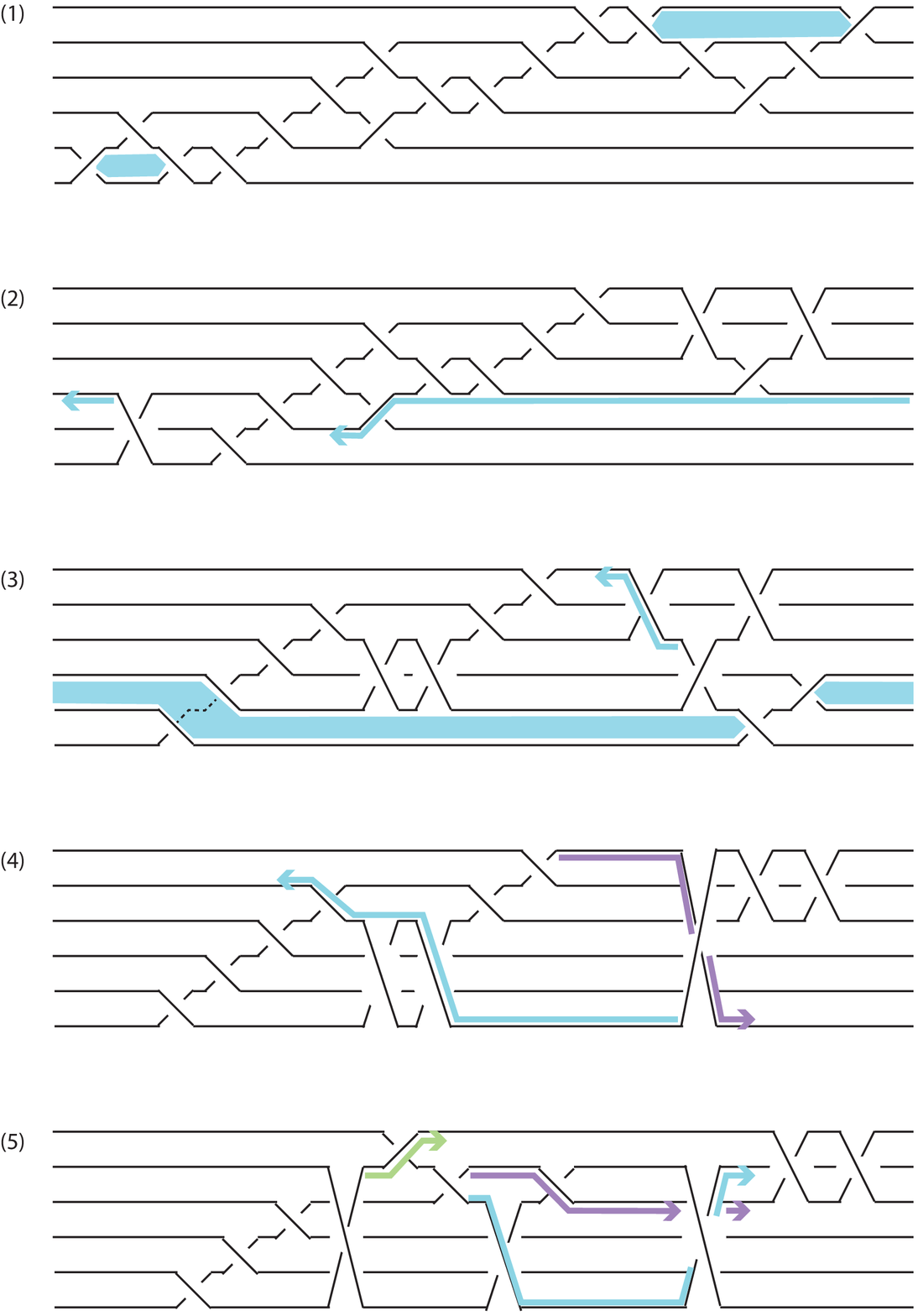}
\includegraphics[width=0.65\textwidth]{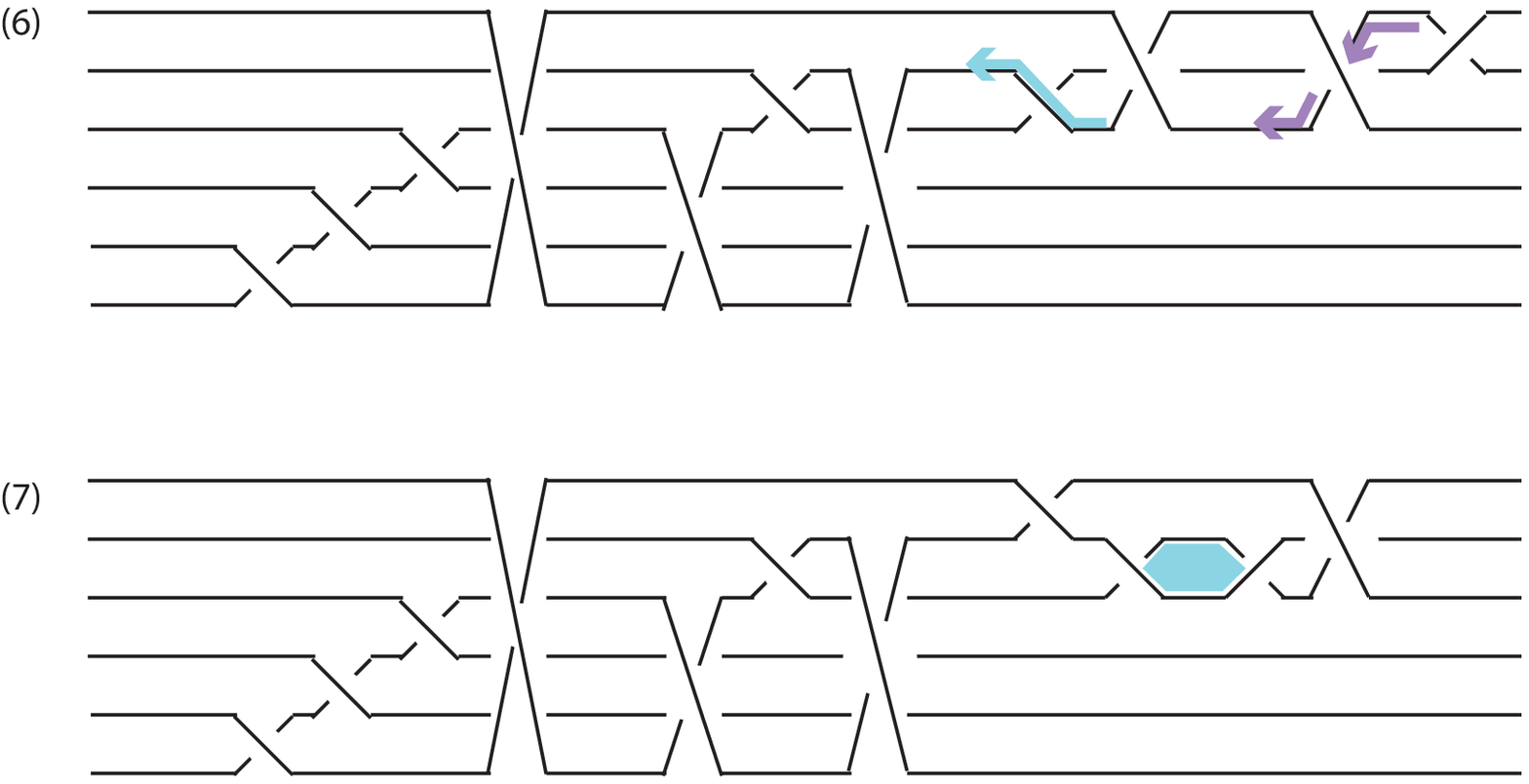}
\caption{Progression from a negatively-flyped braid in standard generators to a strongly quasipositive representative. The arrows indicate the trajectories of slide moves, while the shaded regions indicate the interiors of compression disks. Compression effectively removes a pair of oppositely-signed bands.}\label{fig:flype-example}
\end{figure}

In this way, we can determine that the following thirteen knots, whose strong quasipositivity was previously unknown, are strongly quasipositive. 
$$12_{n148}, 12_{n149}, 12_{n293}, 12_{n321}, 12_{n332}, 12_{n366}, 12_{n404}, 12_{n432}, 12_{n528}, 12_{n642}, 12_{n660}, 12_{n801}, 12_{830n}.$$
Moreover, we are able to find strongly quasipositive words for all the strongly quasipositive knots up to 12 crossing as listed in Section~\ref{subsec-Table-SQP}. 
To aid in completing the list, we use the following Theorems~\ref{1negative} and \ref{separating} that give not only sufficient conditions for strong quasipositivity, but also give explicit strongly quasipositive words.

\section{Almost positive braid links}

In this section we study almost positive braid links. 
Our main result is the following.

\begin{theorem}\label{1negative}
Every almost positive braid link is strongly quasipositive. 
\end{theorem}

We also show a useful condition for strongly quasipositive links. 

\begin{theorem}\label{separating}
Let $n\geq 4$. 
Let $K=w \sigma_1^{-1} w' \sigma_{n-1}^{-1}$ be an $n$-braid where $w$ and $w'$ are positive words in $\{\sigma_1, \ \sigma_{n-1},\ \sigma_{i,j}  \: | \: 2\leq i<j\leq n-1\}.$ 
Then $K$ can be negatively destabilized to a strongly quasipositive braid, or it is conjugate to a strongly quasipositive braid.
\end{theorem}

To prove the above results we observe the following.
\begin{lemma}
\label{lemma:SQP1}
Fix an $i \in\{2, \dots, n\}$. 
Let $K$ (see Figure \ref{fig:lemma1} (a)) be an $n$-braid word
\begin{equation}
\label{eqn:word2}
K = \sigma_{1,i}^{-1}\ X_0\ X_i\ \sigma_{i-1,i}\ X_{i-1}\ \sigma_{i-2,i-1}\ X_{i-2}\cdots \sigma_{1,2}\ X_1
\end{equation}
where 
\begin{itemize}
\item $X_0$ is a (possibly empty) positive word  in $\{\sigma_{a,b}\: | \: 1\leq a < b\leq i-1 \}$. 
\item For $k=1, \dots, i$, $X_{k}$ is a  (possibly empty) positive word  in $\{\sigma_{a,b} \: | \: k\leq a<b\leq n\}$.
\end{itemize}
Then $K$ is a strongly quasipositive braid.
\end{lemma}

\begin{proof}
Let us put
\begin{eqnarray*}
X'_0 &:=& (\sigma_{i-1}\cdots \sigma_{1})^{-1}X_0 (\sigma_{i-1}\cdots \sigma_{1}) \mbox{ and } \\
X'_k &:=& (\sigma_{k-1}\cdots \sigma_{1})^{-1}X_k(\sigma_{k-1}\cdots \sigma_{1}) \mbox{ for } k=1, \dots, i.
\end{eqnarray*}
We view $X'_{0}$ as a positive word  in $\{\sigma_{a,b} \: | \: 2\leq a < b\leq i\}$ obtained from $X_0$ by replacing each  $\sigma_{a,b}$ with $\sigma_{a+1,b+1}$, and $X'_{k}$ as a positive word  in $\{\sigma_{a,b}, \sigma_{1, c} \: | \:  k+1 \leq a<b \leq n,\ k+1 \leq c \leq n\}$ obtained from $X_k$ by replacing each $\sigma_{k,b}$ with $\sigma_{1,b}$.
In particular, $X_1'=X_1$.

Then we have 
\begin{align}
\label{eqn:lemma-start}
K &\sim \sigma_{1,i}^{-1}X_0(X_i \sigma_{i-1,i})\cdots (X_4\sigma_{3,4})(X_3\sigma_{2,3})(X_2\sigma_{1,2})X_1\\
\nonumber &\sim \sigma_{1,i}^{-1}X_0(X_i \sigma_{i-1,i})\cdots (X_4\sigma_{3,4})(X_3\sigma_{2,3})(\sigma_{1,2}X'_{2}) X'_1\\ 
\nonumber&\sim\sigma_{1,i}^{-1}X_0(X_i \sigma_{i-1,i})\cdots (X_4\sigma_{3,4})(\sigma_{2,3} \sigma_{1,2})X'_3 X'_2 X'_1\\
\nonumber&\sim \cdots\\
\label{eqn:lemma-slided} &\sim \sigma_{1,i}^{-1}X_0(\sigma_{i-1,i}\sigma_{i-2,i-1}\cdots \sigma_{1,2})X'_{i}X'_{i-1}\cdots X'_1\\
\nonumber&\sim  \sigma_{1,i}^{-1}(\sigma_{i-1,i}\sigma_{i-2,i-1}\cdots \sigma_{1,2})X'_0 X'_{i}X'_{i-1}\cdots X'_1\\
\label{eqn:lemma-conclusion}&\sim(\sigma_{i-1,i}\cdots \sigma_{2,3}) X'_0X'_{i}X'_{i-1}\cdots X'_1.
\end{align}
Recall that $a \sim b$ means that the braid words $a$ and $b$ represent the same braid in $B_{n}$. 
Hence $K$ is strongly quasipositive.

We may understand the equivalence of (\ref{eqn:lemma-start}) and  (\ref{eqn:lemma-slided}) as the passage (a) $\to$ (b) in Figure \ref{fig:lemma1}, where the bands in $X_k$ slide across the bands $\sigma_{1,2},\ldots,\sigma_{i-1,i}$ in (\ref{eqn:word2}). The equivalence of  (\ref{eqn:lemma-slided}) and (\ref{eqn:lemma-conclusion}) is shown in the transition depicted in Figure~\ref{fig:lemma1} (b) $\to$ (c) where the thick gray arc is tightened across the braid box $X_0$.
\begin{figure}[htbp]
\begin{center}
\includegraphics*[width=100mm]{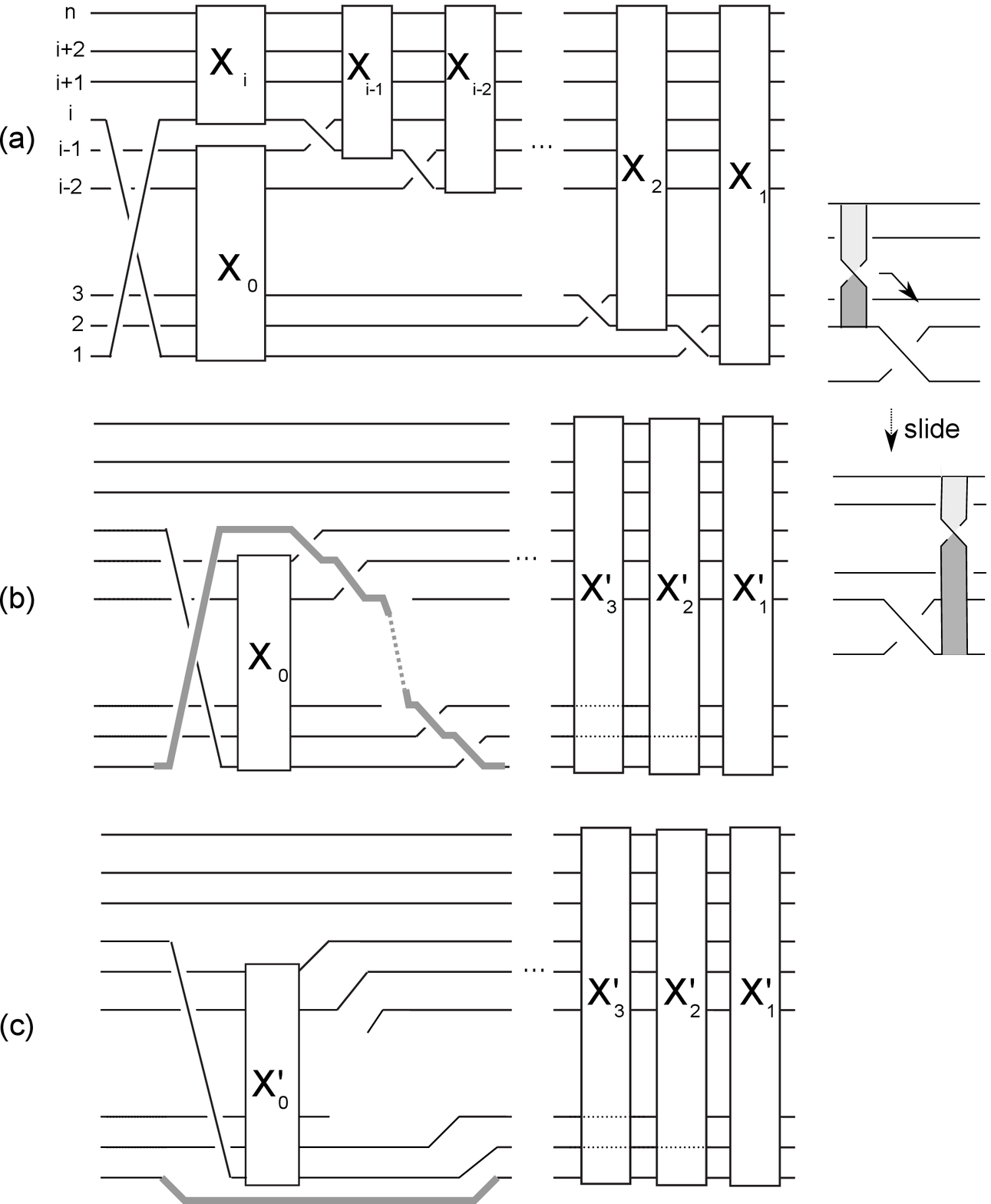}
\caption{Proof of Lemma \ref{lemma:SQP1}. }
\label{fig:lemma1}
\end{center}
\end{figure}

\end{proof}

Using Lemma \ref{lemma:SQP1} we first prove Theorem~\ref{1negative} then Theorem~\ref{separating}.

\begin{proof}[Proof of Theorem~\ref{1negative}]

Let $K$ be an almost positive braid word which contains exactly one negative generator $\sigma_{i}^{-1}$.

If $K$ does not contain $\sigma_i$ then $K$ can be negatively destabilized to a strongly quasipositive $(n-1)$-braid. Therefore, in the following we assume that $K$ contains one $\sigma_i^{-1}$ and at least one $\sigma_i$. Also, we may assume that $K$ contains all the letters $\sigma_{1}, \ldots,\sigma_{n-1}$.

By a cyclic permutation, we may assume that the first letter of $K$ is $\sigma_{i}^{-1}$. 

If $i=1$ then $K$ is of the form $\sigma_{1}^{-1}P \sigma_1Q$ where $P$ is a positive braid word  in $\{\sigma_{2},\ldots,\sigma_{n-1}\}$ and $Q$ is a positive braid word in $\{\sigma_1,\ldots,\sigma_{n-1}\}$. Since $\sigma_{1}^{-1}P\sigma_1$ is strongly quasipositive $K$ is strongly quasipositive. 

Assume that $i\neq 1$. 
Using the braid relations, we may assume that the second letter of $K$ is either $\sigma_{i-1}$ or $\sigma_{i+1}$. 

Let $\Delta=(\sigma_{1}\sigma_{2}\cdots \sigma_{n-1})(\sigma_1\cdots \sigma_{n-2})\cdots (\sigma_{1}\sigma_{2})\sigma_{1}$ be the positive half twist. 
The word $\Delta w \Delta^{-1}$ is equal to the word obtained by $w$ where every letter $\sigma_j$ is replaced with $\sigma_{n-j}$. 
Therefore, taking conjugation by $\Delta$ if necessary, we may further assume that the second letter of $K$ is $\sigma_{i-1}$.

Then after suitable conjugations if necessary, $K$ can be written in the form
\[ K =\sigma_{i}^{-1}\sigma_{i-1}X_{i-1}\sigma_{i-2}X_{i-2}\cdots \sigma_{1}X_{1},\]
where $X_j$ is a positive word  in $\{\sigma_{j},\ldots,\sigma_{n-1}\}$ and at least one of $X_1,\dots,X_{i-1}$ contains $\sigma_i$. 

For $j=1,\dots,i-1$, let $$X'_{j}:=(\sigma_{j-1}\cdots\sigma_{1})^{-1} X_j (\sigma_{j-1}\cdots\sigma_{1}).$$ 
We view $X_j'$ as a positive word in 
$\{\sigma_{1, j+1}, \ \sigma_a \ | \ a=j+1,\dots, n-1\}$ 
which is obtained from $X_j$ by replacing every $\sigma_{j}$ with $\sigma_{1,j+1}$. 
Therefore, $X'_j$ contains $\sigma_{i}$ if and only if $X_j$ contains $\sigma_i$. 
We may change the braid $K$ as
\begin{align*}
K &\sim  \sigma_{i}^{-1}\sigma_{i-1}X_{i-1}\sigma_{i-2}X_{i-2}\cdots \sigma_{1}X_{1}\\
  &\sim  \sigma_{i}^{-1}\sigma_{i-1}\sigma_{i-2} \cdots \sigma_{1}X'_{i-1}X'_{i-2}\cdots X'_{1}\\
  &\sim(\sigma_{i-1}\sigma_{i-2} \cdots \sigma_{1})\sigma_{1,i+1}^{-1}X'_{i-1}X'_{i-2}\cdots X'_{1}\\
 &\approx  \sigma_{1,i+1}^{-1}X'_{i-1}X'_{i-2}\cdots X'_{1}(\sigma_{i-1}\sigma_{i-2} \cdots \sigma_{1}).
\end{align*}
Recall that $a \approx b$ means $a$ and $b$ are  conjugate braids. 

Since $K$ contains at least one $\sigma_i$, the braid word $X'_{i-1}X'_{i-2}\cdots X'_{1}$ contains at least one $\sigma_i$.
This shows that the braid $\sigma_{1,i+1}^{-1}X'_{i-1}X'_{i-2}\cdots X'_{1}(\sigma_{i-1}\sigma_{i-2} \cdots \sigma_{1})$ admits the expression (\ref{eqn:word2}) in Lemma \ref{lemma:SQP1}, so $K$ is strongly quasipositive.
\end{proof}

\begin{proof}[Proof of Theorem~\ref{separating}]
If $K$ does not contain either $\sigma_{n-1}$ or $\sigma_{1}$ then it is easy to check that $K=w \sigma_1^{-1} w' \sigma_{n-1}^{-1}$ can be negatively destabilized twice to an $(n-2)$-braid $w w'$ which is strongly quasipositive. 

If $K$ contains $\sigma_1$ but no $\sigma_{n-1}$ (or contains $\sigma_{n-1}$ but no $\sigma_1$) then $K$ is negatively destabilizable to an ($n-1$)-braid $w \sigma_1^{-1} w'$. By Lemma~\ref{lemma:SQP1}  $K$ is strongly quasipositive. 

Now we assume that $K$ contains both $\sigma_{n-1}$ and $\sigma_{1}$. 
By conjugation we may assume that the first letter of $K$ is $\sigma_{n-1}$ so $K$ is of the form $K=\sigma_{n-1} X \sigma_{n-1}^{-1}Y$ such that exactly one $\sigma_1^{-1}$ is contained in $K$.

If the braid word $X$ contains $\sigma_1^{-1}$ then we have 
\begin{eqnarray*}
K&=& \sigma_{n-1} (P \sigma_{1}^{-1}Q)\sigma_{n-1}^{-1}R \\
&\sim&\sigma_{n-1} P \sigma_{1}^{-1}\sigma_{n-1}^{-1} ( \sigma_{n-1} Q\sigma_{n-1}^{-1})R\\
&\sim&(\sigma_{n-1} P \sigma_{n-1}^{-1}) \sigma_{1}^{-1}( \sigma_{n-1} Q\sigma_{n-1}^{-1})R,
\end{eqnarray*}
where $P$,$Q$ and $R=Y$ are possibly empty positive words in $\{\sigma_1, \ \sigma_{n-1},\ \sigma_{i,j}  \: | \: 2\leq i<j\leq n-1\}.$ 
Then $P':=\sigma_{n-1}P\sigma_{n-1}^{-1}$ and $Q':=\sigma_{n-1}Q\sigma_{n-1}^{-1}$ are positive words in $\{ \sigma_1,\ \sigma_{n-1}, \  \sigma_{i,j} \: | \: 2\leq i<j \leq n~ (j\neq n-1)\} $, which are obtained from $P$ and $Q$, respectively, by replacing $\sigma_{a,n-1}$ with $\sigma_{a,n}$.
Since $K= P' \sigma_1^{-1}Q'R$ contains at least one $\sigma_1$, up to cyclic permutation, the assumption of Lemma \ref{lemma:SQP1} where $i=2$ is satisfied; hence, $K$ is strongly quasipositive.

If the braid word $Y$ contains $\sigma_1^{-1}$ then a parallel argument works. 
\end{proof} 

\section{Almost strongly quasipositive links}

In light of Theorem \ref{1negative}, it is natural and interesting to ask the following: 

\begin{question}
\label{question:almostsqp}
Is every almost strongly quasipositive link quasipositive?
\end{question}

An almost strongly quasipositive link $\mathcal{K}$ satisfies $\delta_3(\mathcal{K})\leq 1$. Conversely, Conjecture \ref{SQPconjecture} asserts that if $\delta_3(\mathcal{K})\leq 1$ then $\K$ is almost strongly quasipositive. 
Therefore, we may extend Question~\ref{question:almostsqp} to the following question: 

\begin{question}\label{question:d=1}
If $\delta_3(\K)=1$ then is $\K$ quasipositive?
\end{question}

Here is some supporting evidence for Questions \ref{question:almostsqp} and \ref{question:d=1}.

\begin{proposition}\label{prop:evidence-d1}
For knots up to 12 crossings in the knot table, the answers to Questions~\ref{question:almostsqp} and \ref{question:d=1} are ``Yes''. 
More precisely, we observe that    
if $\delta_3(\K)=1$ then $\K$ is quasipositive but not strongly quasipositive. 
\end{proposition}

In fact, in Section~\ref{subsec-tableQP} we investigate all knots up to 12 crossings with $\delta_3(\K)=1$. 
Among them, 95 knots are newly discovered to be quasipositive.

The next proposition shows that the answers to Questions~ \ref{question:almostsqp} and \ref{question:d=1} are ``Yes'' for links of braid index 3. 
 
\begin{proposition}\label{prop:3braid-evidence}
If $\beta(\K)=3$ and $\delta_3(\K)=1$ then there exists a $3$-braid representative of $\K$ that is quasipositive but not strongly quasipositive. \end{proposition}

\begin{proof}
Corollary 1.10 in \cite{IK-bennequin-bound} implies that 
there exists a 3-braid representative $K$ of $\K$ in the band generators $\{\sigma_{i,j} \ | \ 1\leq i < j \leq 3\}$ 
such that the associated Bennequin surface $\Sigma_K$ realizes the maximal Euler characteristic $\chi(\K)$ and contains $\delta(\K)=1$ negative band. 
Using \cite[Theorem 6]{X} of Xu, we may assume that $K$ is represented by a braid word having the form $NP$; namely,  
$$K=a_1^{-1} (a_2)^{n_1} (a_3)^{n_2} (a_1)^{n_3} (a_2)^{n_4} (a_3)^{n_5} (a_1)^{n_6}\cdots$$
where $a_1:=\sigma_1, a_2:=\sigma_2, a_3:=\sigma_{1,3}$, and $n_1, \cdots, n_k>0$ and $n_{k+1}=n_{k+2}= \cdots = 0$ for some $k\geq 2$.  
If $k\geq 3$ then $K$ admits a quasipositive word
$$K=(a_1^{-1} a_2 a_1)^{n_1} (a_1^{-1} a_3 a_1)^{n_2} a_1^{n_3-1} a_2^{n_4} (a_3)^{n_5} (a_1)^{n_6}\cdots.$$
If $k=2$ then, up to conjugation, the word can be negatively destabilized into a 2-braid, which is a contradiction. 
\end{proof}

The next theorem shows that the answer to Question \ref{question:almostsqp} is ``Yes'' when $\mathcal{K}$ admits an almost strongly quasipositive 4-braid representative.

\begin{theorem}
\label{theorem:4SQPtoQP}
Suppose that $\K$ has an almost strongly quasipositive 4-braid representative $K$. 
Then one of the following holds. 
\begin{enumerate}
\item[(i)] $K$ is quasipositive (possibly strongly quasipositive, cf. Theorem~\ref{1negative}). 
\item[(ii)] $K$ is a negative stabilization of a strongly quasipositive  3-braid.
\item[(iii)] $K$ admits a negative flype 
and then can be negatively destabilized to a strongly quasipositive braid of braid index at most 3.
\end{enumerate}
\end{theorem}

Theorem~\ref{theorem:4SQPtoQP} is proved in Section~\ref{sec:proof}.

Case (iii) is the most interesting. A negative flype preserves the topological link type and the self-linking number but possibly changes the transverse link type \cite{bm-2}. On the other hand, a positive flype preserves the transverse link type.

This would suggest that the answer to a \emph{transverse link version} of Question \ref{question:almostsqp} ``If a transverse link $\T$ is represented by an almost strongly quasipositive braid then is $\T$ quasipositive?'' is likely to be ``No''.

\section{Alternating or homogeneous quasipositive knots}

From the table of knots of crossing number up to 12 crossings, the following questions naturally appear (cf. \cite{Baader} \cite[Question 7.5]{BBG}):

\begin{question}\label{question:alternating}
{$ $}
\begin{enumerate}
\item[(i)] If $\K$ is alternating (more generally, homogeneous), then $\delta_3(\K) \neq 1$ ?
\item[(ii)] If $\K$ is quasipositive and alternating (more generally, homogeneous) then is $\K$  positive (hence, strongly quasipositive)? 
\end{enumerate} 
\end{question}

We remark that Question \ref{question:alternating} (i) follows from Question \ref{question:alternating} (ii) and Question~\ref{question:d=1}:  
When $\delta_3(\K)=1$, then assuming the truth of Question~\ref{question:d=1}, $\K$ is quasipositive. 
If $\K$ were alternating then Question \ref{question:alternating} (ii) implies $\K$ is strongly quasipositive, which means $\delta_3(\K)=0$, a contradiction. This proves the contrapositive of Question \ref{question:alternating} (i).

Recall the definition of strongly homogeneous links (Definition~\ref{def:strongly-homogeneous}).  
For alternating links, the following characterization of strongly homogeneous links is found by Diao, Hetyei and Liu \cite[Theorem 1.1]{dhl}. 

\begin{proposition}\label{prop:Diao}
Let $\K$ be an alternating link and $D$ be a reduced alternating diagram of $\K$.
The number of Seifert circles in $D$ is equal to the braid index $\beta(\K)$ of $\K$ if and only if for each pair of Seifert circles $s$ and $s'$ of $D$, the number of bands connecting $s$ and $s'$ is not one.  
\end{proposition}
Consequently, many alternating links are strongly homogeneous. 
Proposition \ref{prop:shom-delta} below confirms Question \ref{question:alternating}(i) under strongly homogeneous assumption.

\begin{proposition}
\label{prop:shom-delta}
If $\K$ is a strongly homogeneous link then $\delta_3(\K)\neq 1$. 
\end{proposition}

\begin{proof}
Suppose that $\K$ is strongly homogeneous and $\delta_3(\K)=1$. 

Let $D$ be a homogeneous diagram of $\K$ such that $s(D)=\beta(\K)$. Since $\delta_3(\K)=1$, $\K$ is not a positive link so $D$ cannot be a positive diagram. Hence there is a Seifert circle $s$ of $D$ connected to another Seifert circle $s'$ by negative bands. 
Without loss of generality we may assume that $D$ is reduced.
This means by Proposition~\ref{prop:Diao} the number of negative bands connecting $s$ and $s'$ is greater than one. 
Let $w_{\pm}(D)$ be the number of positive and negative crossings in $D$. 
In particular, $w_-(D)>1$. 

Since $D$ is homogeneous the equation (\ref{eq:useful property}) yields $\chi(\K)=\chi(F_D)=s(D)-w_{+}(D)-w_{-}(D)$. 
By Corollary \ref{cor:diagram-sl}, $SL(\K)= w_{+}(D)-w_{-}(D)-s(D)$.
Then we have $\delta_{3}(\K)= \frac{1}{2}(-SL(\K)-\chi(\K))= w_{-}(D)>1$, which is a contradiction.
\end{proof}

Next we give evidence for Question \ref{question:alternating} (ii). 

\begin{theorem}
\label{theorem:QPalt}
Suppose that $\K$ is strongly homogeneous and quasipositive. 
If every $\beta(\K)$-braid representative of $\K$ is right-veering, then $\K$ is positive.
\end{theorem}

For the definition of right-veering braids see \cite[Section 7]{IK-Qveer}.

\begin{lemma}
\label{lemma:Hom-rv}
Suppose that $\K$ is a non-positive link. 
If there exists a homogeneous link diagram $D$ of $\K$ then $\K$ admits a closed $s(D)$-braid representative which is not right-veering. 
\end{lemma}

The lemma is proved in the end of Section \ref{section:diagram}.

\begin{proof}[Proof of Theorem \ref{theorem:QPalt}]
We prove the contrapositive. 
Assume that $\K$ is a non-positive, strongly homogeneous, quasipositive link. 
By the strongly homogeneous condition, there exists a homogeneous diagram $D$ such that $s(D)=\beta(\K)$.
Then by Lemma \ref{lemma:Hom-rv} and $\K$ admits a $\beta(\K)$-braid representative which is not right-veering. 
\end{proof}

We expect that the assumption of Theorem \ref{theorem:QPalt} is always satisfied for every quasipositive link $\K$.
In the next proposition, we give a sufficient condition for the assumption. 

Transverse links $\T$ and $\T'$ are called \emph{flype-equivalent} \cite{LNS} if their braid representatives are related to each other by negative flypes, positive stabilizations, positive destabilizations and braid isotopies (the last three operations are just the transverse isotopy).
Many known transverse links with the same topological type and the self-linking number are flype equivalent.

\begin{proposition}\label{prop:conditionB}
Let $\K$ be a quasipositive link.
If every pair of transverse links $\T$ and $\T'$ of the topological type $\K$ with $sl(\T)=sl(\T')=SL(\K)$ are flype equivalent, then 
every $\beta(\K)$-braid representative of $\K$ is right-veering. 
\end{proposition}

\begin{proof}

Let $K$ be a $\beta(\K)$-braid representative of $\K$. 
Assume to the contrary that $K$ is not right-veering. 

In \cite[Theorem 1.2]{Hayden} Hayden shows that for a quasipositive link $\K$ there exists a $\beta(\K)$-quasipositive braid representative of $\K$. Let us call it $K'$. 
Both $K$ and $K'$ realize the braid index $\beta(\K)$. 
We view $K$ and $K'$ as transverse links in the standard contact $3$-sphere $(S^3, \xi_{std})$. By \cite{DynnikovPrasolov,LaFountainMenasco} we obtain $sl(K)=sl(K')=SL(\K)$. By the assumption they are related to each other by negative flypes, positive stabilizations, positive destabilizations and braid isotopies.

Let us consider the double branched coverings $(\Sigma_2(K), \xi_2(K))$ (resp. $(\Sigma_2(K'), \xi_2(K'))$) of $(S^3, \xi_{std})$ along $K$ (resp. $K'$). 
Since $K'$ is a quasipositive braid, the contact 3-manifold $(\Sigma_2(K'), \xi_2(K'))$ is Stein fillable \cite[Theorem 1.3]{HKP}. 
Since a negative flype preserves the contact structure on the double branched covering \cite[Theorem 5.9]{HKP}, $(\Sigma_2(K), \xi_2(K))$ is Stein fillable.
On the other hand, since $K$ is not right-veering $(\Sigma_2(K), \xi_2(K))$ is overtwisted \cite[Theorem 5.4]{IK-Branch} 
hence not Stein-fillable. This is a contradiction.
\end{proof}

\section{Canonical surfaces and Bennequin surfaces}
\label{section:diagram}

Alexander's theorem states that every link can be represented by a closed braid. Along with several other methods, Yamada's algorithm \cite{Y} is one way 
to turn a given knot diagram into a braid diagram. 
The advantage of Yamada's algorithm is that it shows that the braid index is equal to the minimum number of Seifert circles. Actually, more is true: the algorithm deforms a canonical Seifert surface into a Bennequin surface as we discuss below (Corollary~\ref{theorem:canonicalBennequin}).

We will restrict our attention to connected link diagrams for simplicity.

Bennequin surfaces, our primary interest, are not canonical Seifert surfaces in general.  For this reason, we introduce a notion of {\em quasi-canonical} Seifert surfaces that includes both Bennequin surfaces and canonical Seifert surfaces, and 
show that every quasi-canonical Seifert surface can be isotoped to a Bennequin surface.

\begin{itemize}
\item[(i)]  
Let $s_1,\dots, s_n \subset \mathbb R^2$ be pairwise disjoint oriented circles. We call them {\em Seifert circles. }
Seifert circles $s$ and $s'$ are called 
\begin{itemize}
\item 
\emph{nested} if $s$ encloses $s'$ in $\mathbb R^2$ or $s'$ encloses $s$ in $\mathbb R^2$.  
\item 
\emph{coherent} (resp. \emph{incoherent}) if $[s]=[s'] \in H_{1}(A)$ (resp. $[s]=-[s'] \in H_{1}(A)$), where $A$ denotes the annulus $A \subset \mathbb R^{2} \cup \{\infty\} = S^2$ bounded by $s$ and $s'$.  
\end{itemize}
\item[(ii)] 
Let $b_1,\dots, b_m \subset \mathbb R^2$ be pairwise disjoint edges with signs $\varepsilon_i\in\{+, -\}$. 
We call them {\em band arcs. }
Each $b_i$ joins distinct Seifert circles and ${\rm Int}(b_i) \cap s_j$ is either empty or transverse.  
If $s$ and $s'$ are joined by edges then we require that $s$ and $s'$ are coherent.
\end{itemize}

Just like the construction of a canonical Seifert surface, given $s_1,\dots, s_n, b_1,\dots, b_m$ we construct an oriented surface $F=F(s_1,\dots, s_n, b_1,\dots, b_m)$ in $\R^{3} = \R^{2} \times \R$: 
Take spanning disks $D_{s_1},\dots, D_{s_n}$ for $s_1,\dots, s_n$.
If $s$ encloses $s'$ then the disk $D_{s'}$ is placed on top of $D_s$. We join the disks by $\varepsilon_i$-twisted bands corresponding to the edges $b_i$.
Whenever an edge intersects a Seifert circle, we regard this as a crossing of the edge over the Seifert circle, and place the corresponding twisted band over the corresponding spanning disk.

\begin{definition}[Quasi-canonical Seifert surface]
The 
Seifert surface $F(s_1,\dots, s_n, b_1,\dots, b_m)$ is called \emph{quasi-canonical} if  (i), (ii) and 
the additional properties (iii) and (iv) are satisfied:
\begin{enumerate}
\item[(iii)] For every pair of band arc $b_i$ and Seifert circle $s_j$, the number of intersections $b_i\cap s_j$ is at most one.
\item[(iv)]
If a band arc $b_i$ connecting $s$ and $s'$ intersects $s''$ then $s, s'$ and $s''$ are pairwise coherent. 
\end{enumerate}
See Figure \ref{fig:quasi_canonical}. 
\end{definition}

\begin{figure}[htbp]
\begin{center}
\includegraphics*[width=110mm]{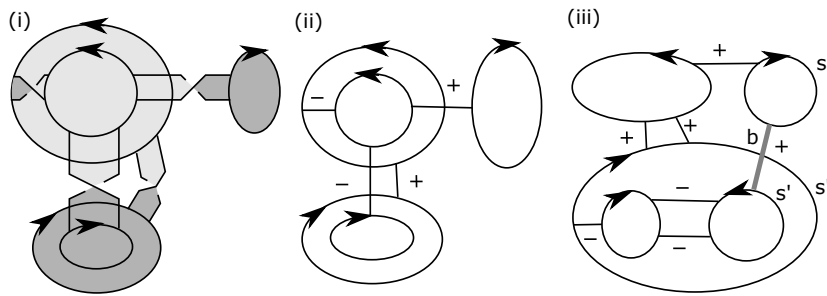}
\caption{(i) A quasi-canonical Seifert surface $F$. 
(ii) Seifert circles and band arcs for $F$.
(iii) A non-quasi-canonical Seifert surface: the band arc $b$ (bold gray) intersects $s''$ which is incoherent to $s$ and $s'$.}
\label{fig:quasi_canonical}
\end{center}
\end{figure}

We observe the following:
\begin{enumerate}
\item
Canonical Seifert surfaces and Bennequin surfaces are quasi-canonical.
\item
A quasi-canonical Seifert surface is canonical if and only if the interior of every band arc is disjoint from all the Seifert circles. 
\item
A quasi-canonical Seifert surface is isotopic to a Bennequin surface if and only if all of the Seifert circles are pairwise coherent (not necessarily pairwise nested).  
\end{enumerate}

One can modify a quasi-canonical Seifert surface into a Bennequin surface by applying isotopy moves called {\em disk-bunching operations}. On the level of link diagram, they are but Yamada's bunching operation of type II \cite{Y}. 
Our disk-bunching operations are refinements of Yamada's which take the disks into account.

\begin{definition}[Disk-bunching operations]
Suppose that $s_1,\dots,s_n, b_1,\dots,b_m$ satisfy the conditions (i) and (ii). 
Let $F:=F(s_1,\dots, s_n, b_1,\dots, b_m)$.
Let $s$ and $s'$ be a pair of incoherent Seifert circles. 
Let $\gamma$ be an oriented arc properly embedded in $\mathbb R^2 \setminus (s_1 \cup\dots\cup s_n \cup b_1\cup \dots\cup b_m)$ that starts from $s$ and ends at $s'$. 
\begin{itemize}
\item Assume that $s$ and $s'$ are not nested.
Replace the disk $D_{s}$ with $D_{s} \cup N(\gamma) \cup D'_{s'}$, where $N(\gamma)$ is a regular neighborhood of $\gamma$ and $D'_{s'}$ is a parallel copy of $D_{s'}$ which lies under $D_{s'}$.
In other words, the operation enlarges $D_{s}$ and slides it under $D_{s'}$. See Figure \ref{fig:bunching}-(i). \\

\item Assume that $s$ and $s'$ are nested and $s'$ encloses $s$. Suppose that $\gamma$ connects a point $p \in s$ to a point $q \in s'$. Let $s^{*} = \gamma \cdot s \cdot \overline{\gamma} \cdot s'$ (read from the right to the left) be a simple closed curve, where $\cdot$ represents concatenation of paths, 
$\overline{\gamma}$ represents $\gamma$ with the reversed orientation, and $s$ and $s'$ are viewed as simple closed curves based at $p$ and $q$, respectively. 
Then we replace the disk $D_s$ with a disk bounded by $s^{*}$ and lying above 
$D_{s'}$.  See Figure \ref{fig:bunching}-(ii). 
We can view this as flipping $D_s$ while fixing the arc $s\setminus \nu(p)$.
\\

\item Assume that $s$ and $s'$ are nested and $s$ encloses $s'$. Suppose that $\gamma$ connects a point $p \in s$ to a point $q \in s'$.  We put $s^{*} = \overline{\gamma} \cdot s' \cdot \gamma \cdot s$
and we replace the disk $D_{s}$ with a disk bounded by $s^{*}$. See Figure \ref{fig:bunching}-(iii). This is a finger move.
\end{itemize}
\end{definition}

\begin{figure}[htbp]
\begin{center}
\includegraphics*[width=95mm, bb=148 480 463 713]{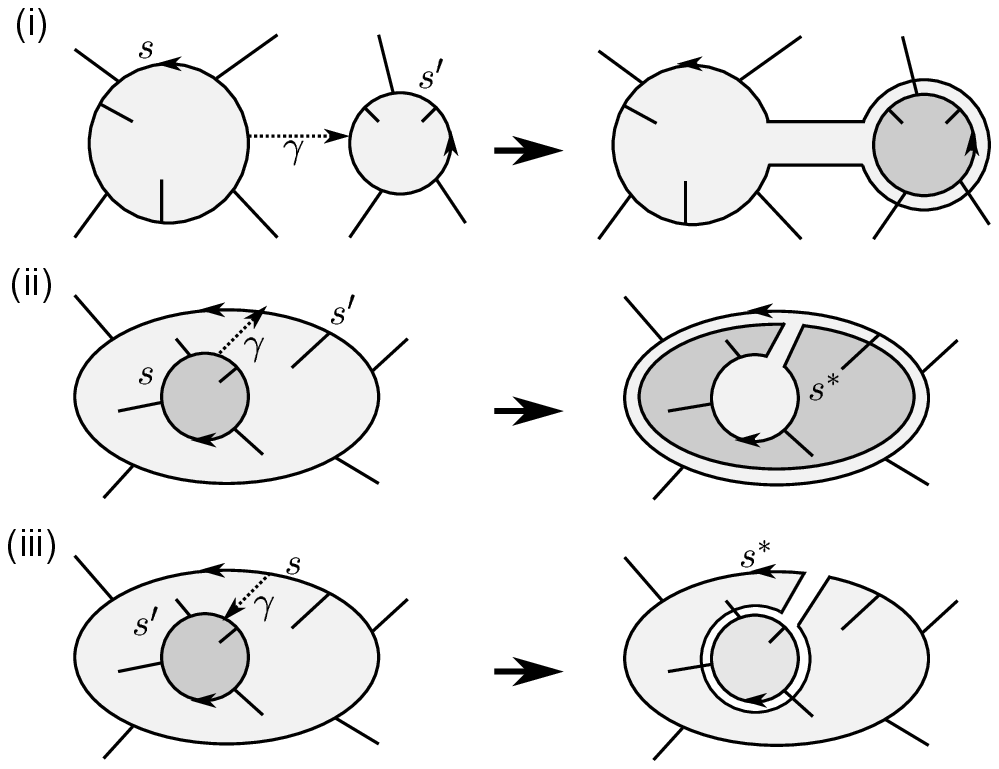}
\caption{Three types of disk-bunching operation.  (The dark gray disk lies above the light gray disk.)}  
\label{fig:bunching}
\end{center}
\end{figure}

\begin{proposition}
Disk-bunching operations are isotopy moves.
If $F$ is quasi-canonical then the resulting surface of disk-bunching operations is also quasi-canonical and $s$ and $s'$ become coherent.  
\end{proposition}

\begin{proof}
The only non-trivial assertion is that of Property (iii). Assume that a band arc $b$ intersects $s^{*}$ more than once, then by property (iii) such a band arc $b$ must intersect both $s$ and $s'$. However, since $s$ and $s'$ are incoherent, such band cannot exist by Property (iv).
\end{proof} 

\begin{theorem}\label{thm:QCSS}
Every quasi-canonical Seifert surfaces can be isotoped to a Bennequin surface without changing the disk-band decomposition structure. (Therefore, the number of Seifert circles and the number of band arcs are preserved.)   
\end{theorem}

\begin{proof}
Applying Disk-bunching operations to a quasi-canonical surface $F$, Theorem 1 of \cite{Y} implies that the Seifert circles become pairwise coherent. 
By the observation (3) the resulting surface is isotopic to a Bennequin surface. 
\end{proof}

\begin{corollary}
\label{theorem:canonicalBennequin}    
Let $D$ be a link diagram. The canonical Seifert surface $F_{D}$ can be isotoped to a Bennequin surface $X$ with the following properties:
\begin{enumerate}
\item[(i)] The boundary $\partial X$ is a closed $s(D)$-braid, where $s(D)$ denotes the number of Seifert circles of $D$. 
\item[(ii)] The number of positive (resp. negative) bands of $X$ is equal to the number of positive (resp. negative) crossings of $D$.

\item[(iii)] The self-linking number of the closed braid $K:=\partial X$ is
$sl(K) = -s(D)+w(D)$, where $w(D)$ denotes the writhe of the diagram $D$.
\end{enumerate}
\end{corollary}

We point out the following corollary which is stronger than the  generalized Jones' conjecture. 
\begin{corollary}
\label{cor:diagram-sl}
Let $D$ be a diagram of a link type $\K$ whose number of Seifert circles is equal to the minimum  number of Seifert circles of all diagrams for $\K$. 
Then its writhe $w(D)$ is an invariant of $\K$ and $w(D)=\beta(\K)+SL(\K)$.
\end{corollary}

\begin{proof}
Since $D$ realizes the minimum number of Seifert circles of all diagrams for $\K$,  
by \cite[Theorem 3]{Y} we have $s(D)=\beta(\K)$.  Thus by Corollary \ref{theorem:canonicalBennequin} (i), the resulting closed braid $K=\partial X$ attains the braid index $\beta(\K)$, in which case $sl(K)=SL(\K)$ is satisfied  \cite{DynnikovPrasolov,LaFountainMenasco}. By Corollary \ref{theorem:canonicalBennequin} (iii), $w(D)=s(D)+sl(K)=\beta(\K)+SL(\K)$.
\end{proof}

Now we are ready to prove Proposition \ref{prop:TFAE} and Lemma \ref{lemma:Hom-rv}:

\begin{proof}[Proof of Proposition \ref{prop:TFAE}]

(ii) $\Rightarrow$ (iii) follows from Corollary \ref{theorem:canonicalBennequin}, and is shown in \cite{N,R} without assuming that $\K$ is homogeneous.

(iii) $\Rightarrow$ (ii) is proven in \cite[Theorem 1]{Baader} which states that a knot is positive if and only if it is homogeneous and strongly quasipositive.

(iii) $\Rightarrow$ (i) is clear. 

We prove (i) $\Rightarrow$ (ii).
Assume that $SL(\K)=-\chi(\K)$.
Let $P_{\K}(v,z)$ be the HOMFLYPT polynomial defined by the skein relation $v^{-1}P_{+}+vP_{-}=zP_{0}$ and the normalization $P_{\sf unknot}(v,z)=1$. 
For any homogeneous link $\K$ we have 
\[SL(\K)+1 \leq \min \deg_{v}P_{\K}(v,z) \leq 1-\chi(\K)\]
where the left inequality follows from the Morton-Franks-Williams inequality \cite{M, FW} and the right inequality follows by \cite[Theorem 4 (b)]{cr}. 
Since $SL(\K)=-\chi(\K)$ we have $$\min\deg_{v}P_{\K}(v,z) = 1-\chi(\K).$$
Again by \cite[Theorem 4 (b)]{cr}, this is equivalent to positivity of $\K$.

Since (iv) implies (iii), it remains to show that if $\K$ is strongly homogeneous, (i) implies (iv). 
Suppose that $D$ is a homogeneous diagram of $\K$ with $s(D)=\beta(\K).$ 
Let $w_\pm(D)$ be the number of positive/negative crossings of $D$. 
Since $D$ is homogeneous the equation (\ref{eq:useful property}) gives
\begin{align*}
\chi(\K)&=\chi(F_D) = s(D)-(w_+(D) + w_-(D)) = \beta(\K) -(w_+(D) + w_-(D)). 
\end{align*}
On the other hand, by the statement (i)  and Corollary~\ref{cor:diagram-sl} we have 
$$\chi(\K) = -SL(\K) = \beta(\K)-w(D)= \beta(\K)-(w_+(D) - w_-(D)).$$ 
Therefore, $w_-(D)=0$ and we know that $D$ is a positive diagram. 
By Corollary \ref{theorem:canonicalBennequin} (ii) $D$ can be isotoped   to a strongly quasipositive braid diagram with $\beta(\K)$ strands.
\end{proof}

\begin{proof}[Proof of Lemma \ref{lemma:Hom-rv}]

Let $D$ be a non-positive, homogeneous link diagram. 
Since $D$ is non-positive
there is a Seifert circle $s$ which is connected to another Seifert circle by negative bands.
Since $D$ is homogeneous, we may assume that on at least one of the sides of $s$, all the band arcs connected to $s$ have negative sign. 
We treat the case that all the band arcs outside of $D_{s}$ have the negative sign. The other case is treated similarly. 

By Corollary~\ref{theorem:canonicalBennequin} the canonical Seifert surface $F_{D}$ can be isotoped to a Bennequin surface $X$ by disk-bunching operations. During this process, by perhaps reversing the orientations of the $\gamma$-arcs of the disk-bunching operations, we may assume that the initial point of every $\gamma$-arc never lies on $s$. This means that the disk $D_{s}$ is not changed by the disk-bunching operations. 

Since $F_D$ is canonical, the observation (2) states that no band arcs intersect $s$ in their interior. 
Assume that the Seifert circle $s$ gives rise to the $i$-th strand in the closed braid $K=\partial X$. Since disk-bunching operations change neither band arcs nor the Seifert circle $s$, there are no band arcs over crossing $s$. 

This means that $K$ contains no bands of the form $\sigma_{k,l}^{\pm 1}$ with $k<i<l$.
Also, for $j<i$ all the bands connecting the $i$-th strand and the $j$-th strand are negative since these come from band arcs of $F_{D}$ which lie outside of $s$. 
 
These two conditions ensure that in the {\em standard} generators, we may write $K$ so that it contains $\sigma_{i-1}^{-1}$ but no $\sigma_{i-1}$. This implies that the braid $K$ is non-right-veering. 
\end{proof}

\section{Links with $\delta_3> 1$} 

It is easy to see that $\delta_3(\K)\geq 2$ 
does not imply quasipositive in general. For example, the figure-eight knot has $\delta_3(\K)=2$ and is not quasipositive.
In Section~\ref{subsec:delta3>1} we give a table of quasipositive knots that have $\delta_3(\K)>1$. Among them, quasipositivity of 60 knots was previously unknown in KnotInfo \cite{KnotInfo}. 

The following proposition shows that any arbitrarily large number can be realized as the defect $\delta_3$ of a quasipositive link.

\begin{proposition}
For any positive integer $\delta$ there exists a quasipositive prime
link $\K$ of braid index $\beta(\K)=3$  such that the defect $\delta_3(\K)=\delta$.
\end{proposition} 

\begin{proof}
Fix an integer $\delta>0$.
Let $\K$ be a knot represented by a 3-braid $K_\delta$ defined by the following. 
$$
K_\delta = \left\{
\begin{array}{lc}
\sigma_1^{-\delta} \sigma_2^2 \ \sigma_{1,3} \ \sigma_1^\delta \sigma_2 & \delta = \mbox{even}\\
\sigma_1^{-\delta} \sigma_2 \ \sigma_{1,3}  \ \sigma_1^\delta \sigma_2^2 &\delta = \mbox{odd}
\end{array}
\right.
$$
Since a closed 3-braid is non-prime if and only if it is conjugate to $\sigma_1^{p}\sigma_2^{q}$ for $p,q>0$ \cite{BirmanMenasco}, we see that $K_{\delta}$ is prime.
Since $K_\delta$ is in Xu's minimal form \cite{X}, its Bennequin surface $\Sigma_{K_\delta}$ has the minimal genus and  
$$\chi(\K)=\chi(\Sigma_{K_\delta})=-1-2\delta.$$
Since $\beta(\K)=\beta(K_\delta)=3$, by the truth of the generalized Jones conjecture we have 
$$SL(\K)=sl(K_\delta)=1.$$
Hence, $\delta(\K) = \delta$. 
\end{proof}

\section{Proof of Theorem \ref{theorem:4SQPtoQP}}\label{sec:proof}

Let $\K$ be a link type in $S^3$. 
Suppose that $\K$ has a 4-braid representative $K$ in the band generators $\{\sigma_{i,j} \ | \ 1 \leq i < j \leq 4\}$ such that $K$ contains only one negative band.  

We may assume that there is a strongly quasipositive braid $\beta$ such that either
\begin{itemize}
\item
(Case 1) $K = \sigma_{1,2}^{-1} \beta$, or 
\item
(Case 2) $K = \sigma_{1,3}^{-1} \beta$.
\end{itemize}

\subsection{Case 1: The negative band is $\sigma_{1,2}^{-1}$}

Let us write $K$ of the form
$$K= \sigma_{1,2}^{-1} R_1 S_1 R_2 \cdots S_{n-1} R_n$$
where
\begin{itemize}
\item
$R_i$ is a SQP word in the letters $\sigma_{2,3}, \sigma_{3,4}$ and $\sigma_{2,4}$. ($R_i$ needn't use all of these letters.)
\item
$S_i$ is a SQP word in the letters $\sigma_{1,3}$ and $\sigma_{1,4}$. ($S_i$ needn't use both of these letters.)
\end{itemize}
We call such a braid representative of $K$ \emph{standard}.
We define the \emph{complexity} of a standard representative 
$K= \sigma_{1,2}^{-1} R_1 S_1 R_2 \cdots S_{n-1} R_n$ by $$(n,-l(R_1),-l(R_2),\ldots, -l(R_n)),$$ where $l(R_i)$ is the length of $R_i$. We compare the complexities of braid representatives by the standard lexicographical ordering.
Suppose that 
$K = \sigma_{1,2}^{-1}R_1 S_1\cdots R_{n-1}S_{n-1}R_n$
is a standard representative of $\K$ that attains the minimum complexity among all the standard representatives.

We use not $l(R_{i})$ but $-l(R_{i})$ in the definition of complexity. Thus, a minimum representative means that we take each $R_i$ to be as long as possible. This counter-intuitive definition of complexity comes from an analogy to Garside (greedy) normal form (see \cite{bkl,bkl2} for normal form of braids).

We will determine the forms of the strongly quasipositive braids $R_{i}$ and $S_{i}$. 
If $n=1$ then $K = \sigma_{1,2}^{-1} R_1$ clearly admits a negative destabilization, so we may assume that $n>1$. 

\begin{claim}
\label{claim:obs0}
$R_{n} \neq 1$ {\rm (the trivial word)}
\end{claim}

\begin{proof}
We observe that $\sigma_{1,2}\sigma_{1,d}\sigma_{1,2}^{-1}\sim \sigma_{2,d}$ for $d=3,4$ and hence $\sigma_{1,2}S_{n-1}\sigma_{1,2}^{-1}$ is a product of $\sigma_{2,3}$ and $\sigma_{2,4}$. If $R_{n} =1$ then 
\begin{align*}
K &\approx S_{n-1} \sigma_{1,2}^{-1} R_1 S_1 R_2 \cdots R_{n-1}\\
&\sim  
(\sigma_{1,2}^{-1}\sigma_{1,2})   S_{n-1} \sigma_{1,2}^{-1} R_1 S_1 R_2 \cdots R_{n-1} \\
&= \sigma_{1,2}^{-1}(\sigma_{1,2} S_{n-1} \sigma_{1,2}^{-1}) R_1 S_1 R_2 \cdots R_{n-1}  \\
&= \sigma_{1,2}^{-1}R_1' S_1 R_2 \cdots R_{n-1}
\end{align*}
where $R'_{1}:=(\sigma_{1,2} S_{n-1} \sigma_{1,2}^{-1}) R_1$. Thus we get a smaller complexity standard representative, which is a contradiction. 
\end{proof}

\begin{claim}
\label{claim:obs1}
If $S_{i}$ contains $\sigma_{1,3}$ and $R_{j}$ contains $\sigma_{2,3}$ 
for some $i<j$, then $K$ is quasipositive.
Similarly, if 
$S_{i}$ contains $\sigma_{1,4}$ and $R_{j}$ contains $\sigma_{2,4}$ 
for some $i<j$, then $K$ is quasipositive.
\end{claim}
\begin{proof}
If $S_{i}$ contains $\sigma_{1,3}$ and $R_{j}$ contains $\sigma_{2,3}$ 
for some $i<j$, then 
$K=\sigma_{1,2}^{-1}\beta \sigma_{1,3}\beta'\sigma_{2,3}\beta'' $
for some strongly quasipositive braids $\beta,\beta'$ and $\beta''$. By replacing $\sigma_{1,3}$ with $\sigma_{2,3}\sigma_{1,2}\sigma_{2,3}^{-1}$ we get 
\begin{align*}
K \sim (\sigma_{1,2}^{-1}\beta \sigma_{2,3}\sigma_{1,2})(\sigma_{2,3}^{-1}\beta'\sigma_{2,3})\beta''.
\end{align*}
This shows that $K$ is quasipositive.
The latter assertion is proven similarly.
\end{proof}

Now we are able to determine $S_i$.

\begin{claim}
\label{claim:S}
If there exists an $i$ such that $\sigma_{1,4} \in S_i$
then $K$ is quasipositive.
\end{claim}
\begin{proof}
First we will show that 
if $S_{n-1}=\sigma_{1,4}^{p}$ then $K$ is quasipositive; therefore, we may assume that $S_{n-1}$ contains $\sigma_{1,3}$:

If  $S_{n-1}= \sigma_{1,4}^{p}$ and 
$R_{n}$ contains $\sigma_{2,4}$ then by Claim \ref{claim:obs1} $K$ is quasipositive. 

Assume $S_{n-1}= \sigma_{1,4}^{p}$ and 
$R_{n}$ does not contain $\sigma_{2,4}$.
Since $R_{n}\neq 1$ by Claim \ref{claim:obs0}, $R_{n}=\sigma_{2,3}^{q}\sigma_{3,4}^{r}$ for some $q,r \geq 0$ with $q+r>0$. We get 
\begin{align*}
K &= \sigma_{1,2}^{-1}R_1 S_1 \cdots S_{n-2} R_{n-1}\sigma_{1,4}^{p}\sigma_{2,3}^{q}\sigma_{3,4}^{r} \\
& \approx \sigma_{1,2}^{-1} (\sigma_{2,4}^p \sigma_{3,4}^{r}R_1)S_1 \cdots S_{n-2} (R_{n-1}\sigma_{2,3}^{q}) \\
&=  \sigma_{1,2}^{-1} R_1' S_1 \cdots S_{n-2} R_{n-1}'
\end{align*}
where $R'_{1}=\sigma_{2,4}^p \sigma_{3,4}^{r}R_1$ and $R'_{n-1}=R_{n-1}\sigma_{2,3}^{q}$. 
So we get a standard representative of smaller complexity, a contradiction.

Now assume that $S_{n-1}$ contains $\sigma_{1,3}$ and $S_i$ contains $\sigma_{1,4}$ for some $i$. 
If $R_n$ contains $\sigma_{2,3}$ or $\sigma_{2,4}$ then by Claim \ref{claim:obs1}, $K$ is quasipositive, done. 
If $R_{n}=\sigma_{3,4}^{r}$ for some $r>0$ then  
\begin{align} \label{eqn:rn}
K &= \sigma_{1,2}^{-1}R_1 S_1\cdots S_{n-1}\sigma_{3,4}^{r} \approx \sigma_{1,2}^{-1} (\sigma_{3,4}^{r}R_1)S_1\cdots S_{n-1}
\end{align}
which contradicts our minimal complexity assumption. 
\end{proof}

We may now assume that for all $i$, $S_{i}=\sigma_{1,3}^{s_i}$ where $s_i>0$.
Next we determine $R_{i}$.
\begin{claim}
\label{claim:R}
If $K$ is not quasipositive, then $n=2$ and 
$$ K = \sigma_{1,2}^{-1} (\sigma_{2,4}^v\sigma_{2,3}^{w})(\sigma_{1,3}^x) (\sigma_{2,3}^y \sigma_{2,4}^z)$$
for some  $v,y \geq 0, x,z>0$ satisfying $v+w>0$. 

\end{claim}

\begin{proof}
By Claims \ref{claim:obs1} and \ref{claim:S}, $R_{i}$ does not contain $\sigma_{2,3}$ for $i>1$.

Let us observe that $R_{n}$ must contain $\sigma_{2,4}$ because otherwise, $R_{n}=\sigma_{3,4}^{r}$ and as we have seen in (\ref{eqn:rn}) we get a standard representative of smaller complexity.

For $1 \leq i < n$, assume that $R_i$ contains $\sigma_{3,4}$. Since $R_{n}$ contains $\sigma_{2,4}$, we have
$$ K= \sigma_{1,2}^{-1}\beta \sigma_{3,4} \beta' \sigma_{1,3} \beta'' \sigma_{2,4}$$
for some strongly quasipositive braids $\beta,\beta'$ and $\beta''$. 
Then
\begin{align*}
K &= \sigma_{1,2}^{-1}\beta \sigma_{3,4} \beta' \sigma_{1,3} \beta'' \sigma_{2,4}\\
 & \sim \sigma_{1,2}^{-1}\beta (\sigma_{3,4} \beta' \sigma_{3,4}^{-1}) \sigma_{3,4}\sigma_{1,3}\sigma_{2,4}(\sigma_{2,4}^{-1}\beta'' \sigma_{2,4})\\
& \sim \sigma_{1,2}^{-1}\beta (\sigma_{3,4} \beta' \sigma_{3,4}^{-1}) \sigma_{1,4}\sigma_{3,4}\sigma_{2,4}(\sigma_{2,4}^{-1}\beta'' \sigma_{2,4})\\
&\sim \sigma_{1,2}^{-1}\beta (\sigma_{3,4} \beta' \sigma_{3,4}^{-1}) \sigma_{1,3} \sigma_{1,2} \sigma_{1,4} 
(\sigma_{2,4}^{-1}\beta'' \sigma_{2,4})\\
& =(\sigma_{1,2}^{-1}\beta (\sigma_{3,4} \beta' \sigma_{3,4}^{-1}) 
 \sigma_{1,3} \sigma_{1,2}) \sigma_{1,4}
(\sigma_{2,4}^{-1}\beta'' \sigma_{2,4})
\end{align*}
so $K$ is quasipositive.

Thus for $1\leq i <n$, $R_{i}$ does not contain $\sigma_{3,4}$.  We conclude 
\[
\begin{cases}
 R_{1} = \sigma_{2,4}^{v} \sigma_{2,3}^{w} \mbox{ for } v, w \geq 0 \mbox{ with } v+w>0 \\
R_{i} = \sigma_{2,4}^{r_i} \mbox{ for some } r_i>0, \mbox{ for all } i=2, \ldots, n-1.
\end{cases}
\]
Next we show that $R_{n}=\sigma_{2,4}^{r_n}$ if $n>2$. Since $R_{n}$ does not contain $\sigma_{2,3}$ it is sufficient to show that $R_{n}$ does not contain $\sigma_{3,4}$, either.
We have already proved that $R_{2}=\sigma_{2,4}^{p}$. 
Hence if $R_{n}$ contains $\sigma_{3,4}$ then $K=\sigma_{1,2}^{-1}\beta\sigma_{1,3}\beta' \sigma_{2,4}\beta''\sigma_{3,4}\beta'''$ for some strongly quasipositive braids $\beta,\beta',\beta''$ and $\beta'''$. Then
\begin{align*}
K &= \sigma_{1,2}^{-1}\beta \sigma_{1,3} \beta'\sigma_{2,4} \beta'' \sigma_{3,4} \beta''' \sim(\sigma_{1,2}^{-1}\beta \sigma_{2,3}\sigma_{1,2})(\sigma_{2,3}^{-1} \beta' \sigma_{3,4}\sigma_{2,3})\sigma_{3,4}^{-1}\beta''\sigma_{3,4}\beta'''
\end{align*}
so $K$ is quasipositive, a contradiction.  

Finally we show that $R_1=\sigma_{2,4}^{r_1}$ $(r_1>0)$ if $n>2$. It remains to show that $R_{1}$ does not contain $\sigma_{2,3}$. If $R_{1}$ contain $\sigma_{2,3}$ and $n>2$, then $K=\sigma_{1,2}^{-1} \beta_0 \sigma_{2,3}\beta_1 \sigma_{2,4}\beta_2\sigma_{1,3}\beta_3\sigma_{2,4}\beta_4$ for some strongly quasipositive braids $\beta_0,\ldots,\beta_4$ (Here $\sigma_{2,3},\sigma_{2,4},\sigma_{1,3}$ and the last $\sigma_{2,4}$ come from $R_{1}$, $R_{2}$, $S_{2}$ and $R_{3}$, respectively). Then 
\begin{align*}
K&= \sigma_{1,2}^{-1} \beta_0 \sigma_{2,3}\beta_1 \sigma_{2,4}\beta_2\sigma_{1,3}\beta_3\sigma_{2,4}\beta_4\\
 & \sim\sigma_{1,2}^{-1} \beta_0 (\sigma_{2,3}\beta_1\sigma_{2,3}^{-1})\sigma_{2,3} \sigma_{3,4}\sigma_{2,3}\sigma_{3,4}^{-1}
\beta_{2}\sigma_{2,3}\sigma_{1,2}\sigma_{2,3}^{-1}\beta_{3}
\sigma_{3,4}\sigma_{2,3}\sigma_{3,4}^{-1}\beta_4\\
&\sim \sigma_{1,2}^{-1}\sigma_{3,4}\sigma_{3,4}^{-1}\beta_0 (\sigma_{2,3}\beta_1\sigma_{2,3}^{-1})\sigma_{3,4} \sigma_{2,3}
\beta_{2}\sigma_{2,3}\sigma_{1,2}\sigma_{2,3}^{-1}\beta_{3}
\sigma_{3,4}\sigma_{2,3}\sigma_{3,4}^{-1}\beta_4\\
&\sim \{\sigma_{3,4}[\sigma_{1,2}^{-1}(\sigma_{3,4}^{-1}\beta_0 (\sigma_{2,3}\beta_1\sigma_{2,3}^{-1})\sigma_{3,4}) \sigma_{2,3}
\beta_{2}\sigma_{2,3}\sigma_{1,2}](\sigma_{2,3}^{-1}\beta_{3}
\sigma_{3,4}\sigma_{2,3})\sigma_{3,4}^{-1}\}\beta_4
\end{align*}
so $K$ is quasipositive, a contradiction.

So far, we have obtained that 
\begin{eqnarray*}
(n=2) && R_1 = \sigma_{2,4}^v \sigma_{2,3}^w \mbox{ and } 
R_2 = \sigma_{3,4}^y \sigma_{2,4}^z \mbox{ for some } v, w, y \geq 0, z>0, v+w>0\\
(n\geq 3) && R_i = \sigma_{2,4}^{r_i} \mbox{ for all } i=1,\cdots, n.
\end{eqnarray*}

However, the braid $
K= \sigma_{1,2}^{-1}\sigma_{2,4}^{r_1}\sigma_{1,3}^{s_1}\cdots \sigma_{2,4}^{r_{n-1}}\sigma_{1,3}^{s_{n-1}}\sigma_{2,4}^{r_{n}}$
$(r_{i},s_{i}>0, n>2)$ is quasipositive by the following computation. 
\begin{align*}
K &= \sigma_{1,2}^{-1}\sigma_{2,4}^{r_1}\sigma_{1,3}^{s_1}\cdots\sigma_{2,4}^{r_{n}} & \\
 & \sim \sigma_{1,4}^{r_1}\sigma_{1,2}^{-1}\sigma_{1,3}^{s_1}\cdots\sigma_{2,4}^{r_{n}} & (\because  \sigma_{1,2}^{-1}\sigma_{2,4}\sigma_{1,2} \sim\sigma_{1,4})\\
 & \sim \sigma_{1,4}^{r_1} \sigma_{1,3} \sigma_{2,3}^{-1}\sigma_{1,3}^{s_{1}-1}\cdots\sigma_{2,4}^{r_{n}} & 
(\because  \sigma_{1,2}^{-1}\sigma_{1,3} \sim\sigma_{1,3}\sigma_{2,3}^{-1})\\
 & \sim \sigma_{1,4}^{r_1} \sigma_{1,3} (\sigma_{2,3}^{-1}\sigma_{1,3}^{s_{1}-1}\sigma_{2,3})
\sigma_{2,3}^{-1}\sigma_{2,4}^{r_2}\cdots\sigma_{2,4}^{r_{n}} & \\
&\sim \sigma_{1,4}^{r_1} \sigma_{1,3} \sigma_{1,2}^{s_{1}-1} 
\sigma_{2,4}\sigma_{3,4}^{-1}\sigma_{2,4}^{r_2 -1}\cdots\sigma_{2,4}^{r_{n}} & (\because  \sigma_{2,3}^{-1}\sigma_{2,4}  \sim \sigma_{2,4}\sigma_{3,4}^{-1})\\
& \sim \sigma_{1,4}^{r_1} \sigma_{1,3} \sigma_{1,2}^{s_{1}-1} 
\sigma_{2,4}(\sigma_{3,4}^{-1}\sigma_{2,4}^{r_{2}-1}\sigma_{3,4})\sigma_{3,4}^{-1}\sigma_{1,3}^{s_3}\cdots \sigma_{2,4}^{r_{n}} & \\
& \sim \sigma_{1,4}^{r_1} \sigma_{1,3} \sigma_{1,2}^{s_{1}-1}
\sigma_{2,4}\sigma_{2,3}^{r_{2}-1} 
\sigma_{1,3}\sigma_{1,4}^{-1}\sigma_{1,3}^{s_2-1}\cdots \sigma_{2,4}^{r_{n}} & 
(\because \sigma_{3,4}^{-1}\sigma_{1,3} \sim \sigma_{1,3}\sigma_{1,4}^{-1})\\
& \sim \sigma_{1,4}^{r_1 -1}[\sigma_{1,4}\sigma_{1,3} \sigma_{1,2}^{s_{1}-1}
\sigma_{2,4}\sigma_{2,3}^{r_{2}-1}
\sigma_{1,3}\sigma_{1,4}^{-1}]\sigma_{1,3}^{s_2-1}\sigma_{2,4}^{r_3} \cdots \sigma_{2,4}^{r_{n}} 
\end{align*}

This completes the proof of the claim. 
\end{proof}

\begin{claim}
If $ K = \sigma_{1,2}^{-1} (\sigma_{2,4}^v\sigma_{2,3}^{w})(\sigma_{1,3}^x) (\sigma_{3,4}^y \sigma_{2,4}^z)$ for $v,y \geq 0, x,z>0$, and $v+w>0$, then $K$ admits a negative flype. After the negative flype we get a negatively destabilizable 4-braid. Thus the braid index of $\K$ has $\beta(\K)\leq 3$. 

\end{claim}
\begin{proof}
One can modify the braid $K$ as follows. 
\begin{align*}
K&= \sigma_{2,4}^v\sigma_{2,3}^{w}\sigma_{1,3}^x \sigma_{3,4}^y \sigma_{2,4}^z \sigma_1^{-1} \\
&\sim 
\sigma_{2,4}^v \sigma_{2,3}^{w}\sigma_{1,3}^x \sigma_{3,4}^y \sigma_{2,4}^{z-1} \sigma_{1,4}^{-1}  \sigma_{2,4} \\
&\sim 
\sigma_{2,4}^v \sigma_{2,3}^{w} \sigma_{1,3}^x \sigma_{3,4}^y \sigma_{2,4}^{z-1} (\sigma_{3,4} \sigma_{1,3}^{-1} \sigma_{3,4}^{-1})  \sigma_{2,4} \\
&\sim 
\sigma_{2,4}^v  \sigma_{2,3}^{w}\sigma_{1,3}^x \sigma_{3,4}^{y+1} \sigma_{2,3}^{z-1} \sigma_{1,3}^{-1} \sigma_{2,3}  \sigma_{3,4}^{-1} \\
&\sim 
\sigma_{2,3}^v \sigma_{3,4}^{-1} \sigma_{2,3}^{w} \sigma_{1,3}^x \sigma_{3,4}^{y+1} \sigma_{2,3}^{z-1} \sigma_{1,3}^{-1} \sigma_{2,3}   \\
\end{align*}
This closed 4-braid $K$ admits a negative flype. By performing a negative flype, we get another closed braid representative $K'$ given by
\begin{align*}
K'&= 
\sigma_{2,3}^v \sigma_{3,4}^{y+1} \sigma_{2,3}^{w} \sigma_{1,3}^x \sigma_{3,4}^{-1} \sigma_{2,3}^{z-1} \sigma_{1,3}^{-1} \sigma_{2,3} \sim \sigma_{2,3}^v \sigma_{3,4}^y \sigma_{2,4}^{w}\sigma_{2,3}^{z-1}\sigma_{1,3}^{-1}  \sigma_{3,4}^x  \sigma_{2,3}.  
\end{align*} 
The last braid admits a negative destabilization to a strongly quasipositive 3-braid.
\end{proof}

\subsection{Case 2: The negative band is $\sigma_{1,3}^{-1}$} 

Let $K=\sigma_{1,3}^{-1}\beta$. If $\beta$ contains $\sigma_{1,3}$ then $K$ is clearly quasipositive so we may assume that $\beta$ does not contain $\sigma_{1,3}$.

To begin with, we observe that the following relations hold:
\begin{equation}
\label{eqn:A1}
\sigma_{1,3}^{-1}\sigma_{2,3} \sim \sigma_{2,3}\sigma_{1,2}^{-1}, \quad \sigma_{1,3}^{-1}\sigma_{1,4}\sim\sigma_{1,4}\sigma_{3,4}^{-1}
\end{equation}
\begin{equation}
\label{eqn:A2}
\sigma_{1,3}^{-1}\sigma_{1,2}\sim\sigma_{2,3}\sigma_{1,3}^{-1}, \quad \sigma_{1,3}^{-1}\sigma_{3,4}\sim\sigma_{1,4}\sigma_{1,3}^{-1}.
\end{equation}
By these relations we get
\begin{claim}
If $\beta$ does not contain $\sigma_{2,4}$ then $K$ is either 
quasipositive, negatively destabilizable to a strongly quasipositive 3-braid, or reduces to Case 1.
\end{claim}

\begin{proof}
Assume that $\beta$ does not contain $\sigma_{2,4}$. By  (\ref{eqn:A1}) and (\ref{eqn:A2}), if $\beta$ contains either $\sigma_{2,3}$ or $\sigma_{1,4}$, then we can write $K$ so that it contains a negative band $\sigma_{1,2}^{-1}$ or $\sigma_{3,4}^{-1}$. Indeed, we may slide any $\sigma_{1,2}$ or $\sigma_{3,4}$ left of $\sigma_{1,3}^{-1}$ using (\ref{eqn:A2}), and then the first letter of $\beta$ may be assumed to be either $\sigma_{2,3}$ or $\sigma_{1,4}$, in which case we apply (\ref{eqn:A1}).

Now assume that $\beta$ does not contain $\sigma_{2,4}$, $\sigma_{2,3}$ or $\sigma_{1,4}$; that is,  $\beta=\sigma_{1,2}^{x}\sigma_{3,4}^{y}$ for some $x,y\geq 0$. 
Then
$$K= \sigma_{1,3}^{-1}\sigma_{1,2}^{x}\sigma_{3,4}^{y} \sim \sigma_{2,3}^{x}\sigma_{1,3}^{-1}\sigma_{3,4}^{y}$$
admits a negative destabilization to a SQP 3-braid.
\end{proof}

Thus we may write $K=\sigma_{1,3}^{-1}\beta$ such that $\beta$ is a strongly quasipositive braid that contains $\sigma_{2,4}$.
 Moreover, by applying  (\ref{eqn:A1}) and (\ref{eqn:A2}) we may actually write
$$
K= \sigma_{1,3}^{-1}\sigma_{2,4}\beta
$$
for some strongly quasipositive braid $\beta$.
As in Case 1 let us write $K$ as
$$
K= \sigma_{1,3}^{-1}\sigma_{2,4}R_1 S_1 \cdots R_{n}
$$
where
\begin{itemize}
\item
$R_i$ is a SQP word in the letters $\sigma_{2,3}, \sigma_{3,4}$ and $\sigma_{2,4}$. ($R_i$ needn't use all of these letters.)
\item
$S_i$ is a SQP word in the letters $\sigma_{1,2}$ and $\sigma_{1,4}$. ($S_i$ needn't use both of these letters.)
\end{itemize}
We call such a braid representative \emph{standard}, and the complexity of standard representative is defined similarly to that of Case 1.

Let us take a standard representative of $K$ that attains the minimum complexity among all the standard representatives.

\begin{claim}
\label{claim:obscase2}
If $R_i$ contains $\sigma_{3,4}$ for some $i>1$, then $K$ is quasipositive.
\end{claim}
\begin{proof}
Assume $R_i$ contains $\sigma_{3,4}$ for some $i>1$. 
Let $\sigma_{1,d}$ $(d=2,4)$ be the last letter of the strongly quasipositive braid  $S_{i-1}$, so that $S_{i-1}=S'\sigma_{1,d}$.
If $R_i$ contains $\sigma_{3,4}$ then 
$K=\sigma_{1,3}^{-1}\sigma_{2,4}\beta' \sigma_{1,d}\beta''\sigma_{3,4}\beta'''$
for some strongly quasipositive braids $\beta',\beta''$ and $\beta'''$.

If $d=2$, by replacing $\sigma_{1,3}^{-1}$ and $\sigma_{2,4}$ with $\sigma_{2,3}\sigma_{1,2}^{-1}\sigma_{2,3}^{-1}$ and $\sigma_{3,4}\sigma_{2,3}\sigma_{3,4}^{-1}$ respectively we get
\begin{align*}
K&= (\sigma_{2,3}\sigma_{1,2}^{-1}\sigma_{2,3}^{-1})
(\sigma_{3,4}\sigma_{2,3}\sigma_{3,4}^{-1})\beta' \sigma_{1,2}\beta''\sigma_{3,4}\beta'''\\
&\sim \sigma_{2,3}\sigma_{1,2}^{-1}(\sigma_{3,4}^{-1}\sigma_{3,4})\sigma_{2,3}^{-1}\sigma_{3,4}
\sigma_{2,3}\sigma_{3,4}^{-1}\beta' \sigma_{1,2}\beta''\sigma_{3,4}\beta'''\\
&\sim \sigma_{2,3}\{\sigma_{3,4}^{-1}[\sigma_{1,2}^{-1}(\sigma_{3,4}\sigma_{2,3}^{-1}\sigma_{3,4}\sigma_{2,3}\sigma_{3,4}^{-1})\beta' \sigma_{1,2}]\beta''\sigma_{3,4}\}\beta'''.
\end{align*}
This shows $K$ is quasipositive.

Similarly, if $d=4$ then by replacing $\sigma_{1,4}$ with $\sigma_{3,4}\sigma_{1,3}\sigma_{3,4}^{-1}$
\begin{align*}
K&= \sigma_{1,3}^{-1}\sigma_{2,4}\beta' (\sigma_{3,4}\sigma_{1,3}\sigma_{3,4}^{-1}) \beta''\sigma_{3,4}\beta''' \sim (\sigma_{1,3}^{-1}\sigma_{2,4}\beta' \sigma_{3,4}\sigma_{1,3})(\sigma_{3,4}^{-1}\beta''\sigma_{3,4})\beta'''.
\end{align*}
This shows $K$ is quasipositive.
\end{proof}

\begin{claim}
Assume that $S_{i}$ contains $\sigma_{1,2}$ for some $i>0$ and that $R_{1}$ contains $\sigma_{3,4}$. Then $K$ is quasipositive.
\end{claim}
\begin{proof}
Since $S_i$ contains $\sigma_{1,2}$ and $R_1$ contains $\sigma_{3,4}$, we may write $K$ as  $K=\sigma_{1,3}^{-1}\sigma_{2,4}\beta'\sigma_{3,4}\beta''\sigma_{1,2}\beta'''$
for some strongly quasipositive braids $\beta',\beta''$ and $\beta'''$.
By replacing $\sigma_{1,3}^{-1}$ and $\sigma_{2,4}$ with $\sigma_{2,3}\sigma_{1,2}^{-1}\sigma_{2,3}^{-1}$ and $\sigma_{3,4}\sigma_{2,3}\sigma_{3,4}^{-1}$ respectively we get
\begin{align*}
K&= (\sigma_{2,3}\sigma_{1,2}^{-1}\sigma_{2,3}^{-1})(\sigma_{3,4}\sigma_{2,3}\sigma_{3,4}^{-1})
\beta' \sigma_{3,4}\beta''\sigma_{1,2}\beta'''\\
&\sim \sigma_{2,3}[\sigma_{1,2}^{-1}(\sigma_{2,3}^{-1}\sigma_{3,4}\sigma_{2,3})
(\sigma_{3,4}^{-1}\beta' \sigma_{3,4})\beta''\sigma_{1,2}]\beta'''.
\end{align*}
This shows $K$ is quasipositive.
\end{proof}

Thus, we conclude if $S_{i}$ contains $\sigma_{1,2}$ for some $i>0$ then either $K$ is quasipositive, or, $R_i$ does not contain $\sigma_{3,4}$ for every $i$.

\begin{claim}
If $S_i$ contains $\sigma_{1,2}$ for some $i>0$ and $R_{i}$ does not contain $\sigma_{3,4}$ for every $i$, then $K$ is either negatively destabilzable or quasipositive.
\end{claim}
\begin{proof}
By assumption we may write $K$ as
$$
K = \sigma_{1,3}^{-1} \sigma_{2,4} \beta_0 \sigma_{1,2} \beta_1
$$
where $\beta_0$ and $\beta_1$ are strongly quasipositive braids such that 
\begin{itemize}
\item $\beta_0$ does not contain $\sigma_{3,4}$.
\item $\beta_{1}$ does not contain $\sigma_{3,4}$ or $\sigma_{1,2}$.
\end{itemize} 

If neither $\beta_0$ nor $\beta_1$ contain $\sigma_{2,3}$, then the first letter $\sigma_{1,3}^{-1}$ is the unique band that touches the 3rd strand so $K$ is negatively destabilizable to a strongly quasipositive 3-braid.

If $\beta_0$ contains $\sigma_{2,3}$ then we have $\beta_0=\beta' \sigma_{2,3}\beta''$ for some strongly quasipositive braids $\beta'$ and $\beta''$. Hence
\begin{align*}
K &=\sigma_{1,3}^{-1} \sigma_{2,4} \beta' \sigma_{2,3}\beta'' \sigma_{1,2} \beta_1\\
 & \sim (\sigma_{2,3} \sigma_{1,2}^{-1}\sigma_{2,3}^{-1}) \sigma_{2,4} \beta' \sigma_{2,3}\beta'' \sigma_{1,2} \beta_1\\
 & \sim \sigma_{2,3}[\sigma_{1,2}^{-1}(\sigma_{2,3}^{-1} \sigma_{2,4} \beta' \sigma_{2,3})\beta'' \sigma_{1,2}] \beta_1
\end{align*}
so $K$ is quasipositive.

Thus we may assume that $\beta_0$ does not contain $\sigma_{2,3}$ but $\beta_1$ contains at least one $\sigma_{2,3}$. We may write $\beta_1=\beta' \sigma_{2,3}\beta''$ where $\beta'$ and $\beta''$ are strongly quasipositive braids such that 
\begin{itemize}
\item $\beta'$ is a word in the letters $\sigma_{1,4}$ and $\sigma_{2,4}$.
That is, $\beta' = \sigma_{2,4}^p\sigma_{1,4}^q$ for some $p,q\geq 0$. 
\item $\beta''$ is a word in the letters $\sigma_{1,4}$, $\sigma_{2,4}$ and $\sigma_{2,3}$.
\end{itemize}

If $\beta''=1$ then
\begin{align*}
K &=\sigma_{1,3}^{-1} \sigma_{2,4} \beta_0 \sigma_{1,2}\beta'\sigma_{2,3}\\
 & \approx \sigma_{2,4} \beta_0 \sigma_{1,2}\beta'\sigma_{2,3}\sigma_{1,3}^{-1}\\
 &\sim \sigma_{2,4} \beta_0 \sigma_{1,2}\beta'\sigma_{1,3}^{-1}\sigma_{1,2}. 
\end{align*}
Since both $\beta_0$ and $\beta'$ contain neither $\sigma_{2,3}$ nor $\sigma_{3,4}$, the negative band $\sigma_{1,3}^{-1}$ is the unique band that touches the third strand. Thus $K$ is negatively destabilizable. 

So we may assume $\beta''\neq 1$. Suppose $\beta''$ does not contain $\sigma_{2,4}$. Then $\beta''$ is a word in the letters $\sigma_{2,3}$ and $\sigma_{1,4}$. After sliding all such $\sigma_{1,4}$ into $\beta'$ we may write
\begin{align*}
K &= \sigma_{1,3}^{-1}\sigma_{2,4}\beta_0\sigma_{1,2}\sigma_{2,4}^x\sigma_{1,4}^y\sigma_{2,3}^z\\
&\approx \sigma_{2,3}^z\sigma_{1,3}^{-1}\sigma_{2,4}\beta_0\sigma_{1,2}\sigma_{2,4}^x\sigma_{1,4}^y\\
&\sim\sigma_{1,3}^{-1}\sigma_{1,2}^z\sigma_{2,4}\beta_0\sigma_{1,2}\sigma_{2,4}^x\sigma_{1,4}^y.
\end{align*}
The first letter $\sigma_{1,3}^{-1}$ is the only band touching the third strand, and we may negatively destabilize $K$ to a strongly quasipositive 3-braid.

Finally, we may assume that $\beta''\neq 1$ contains a $\sigma_{2,4}.$ If the first letter of $\beta''$ is $\sigma_{1,4},$ we may slide it into $\beta'$. If the first letter of $\beta''$ is a $\sigma_{2,4}$, then we may apply the relation $\sigma_{2,3}\sigma_{2,4} = \sigma_{2,4}\sigma_{3,4}$, contradicting the assumption that $\beta_1$ does not contain a $\sigma_{3,4}$. If $\beta''$ starts with $\sigma_{2,3}$ then by possibly sliding some $\sigma_{1,4}$ into $\beta'$ we may assume that $\beta'' = \sigma_{2,3}^x\sigma_{2,4}\beta'''$ for some strongly quasipositive braid $\beta'''$ and $x>0$. But we may again apply the above relation to produce a $\sigma_{3,4}$ in $\beta_1$. 
\end{proof}

Thus we may assume that $S_{i}=\sigma_{1,4}^{s_i}$ $(s_{i}>0)$, and that $R_{i}$ does not contain $\sigma_{3,4}$ for $i>1$. 

\begin{claim}
If $n > 1$, then $K$ is quasipositive.
\end{claim}

\begin{proof}
We may write $R_{2}=\sigma_{2,4}^{x}\sigma_{2,3}^{y}$ for some $x,y\geq 0$ such that $x+y>0$.

If $x>0$ then 
\begin{align*}
K &= \sigma_{1,3}^{-1}\sigma_{2,4}R_1 \sigma_{1,4}^{s_1}\sigma_{2,4}^{x}\sigma_{2,3}^{y} S_2 \cdots R_n\\
& \sim  \sigma_{1,3}^{-1}\sigma_{2,4}(R_1\sigma_{2,4}) \sigma_{1,2}^{s_1} (\sigma_{2,4}^{x-1}\sigma_{2,3}^{y} )S_2\cdots R_n.
\end{align*}
We get a standard representative with smaller complexity since $-l(R_1) > -l(R_1 \sigma_{2,4})$. 

If $x=0$ then $y>0$ and
\begin{align*}
K &= \sigma_{1,3}^{-1}\sigma_{2,4}R_1 \sigma_{1,4}^{s_1}\sigma_{2,3}^{y} S_2\cdots R_n\\
& \sim  \sigma_{1,3}^{-1}\sigma_{2,4}(R_1\sigma_{2,3}^{y})\sigma_{1,4}^{s_1}S_2 \cdots R_n
\end{align*}
so we get a standard representative with smaller complexity.
\end{proof}

Finally it remains to check the case $R_{2}$ is empty so  $K=\sigma_{1,3}^{-1}\sigma_{2,4}R_1\sigma_{1,4}^{s_1}$. 

\begin{claim}
If $K=\sigma_{1,3}^{-1}\sigma_{2,4}R_1\sigma_{1,4}^{s_1}$ then $K$ is negatively destabilizable to a strongly quasipositive 3-braid. 
\end{claim}
\begin{proof}
$K =\sigma_{1,3}^{-1}\sigma_{2,4}R_1\sigma_{1,4}^{s_1} \sim \sigma_{2,4}R_1\sigma_{1,4}^{s_1}\sigma_{1,3}^{-1} \sim \sigma_{2,4}R_1\sigma_{1,3}^{-1}\sigma_{3,4}^{s_1}$
so it is negatively destabilizable.
\end{proof}

\section{Tables of strongly quasipositive and quasipositive knots up to 12 crossings}
\label{sec:tables}

In the section we list all strongly quasipositive (resp. quasipositive) knots up to 12 crossings (except for $12n_{239}$ and $12n_{512}$) along with their strongly quasipositive (resp. quasipositive) braid representatives. 
We could not determine whether $12n_{239}$ and $12n_{512}$ are quasipositive or not. 

In the table, we express a braid word $\sigma_{i_1}^{\pm 1}\sigma_{i_2}^{\pm 2} \cdots$ by a sequence of integers $\pm i_1,\pm i_2,\cdots$.
Also, in the column ``comment'',
\begin{itemize}
\item[--] (mirror) means that a strongly quasipositive braid or quasipositive braid in the table represents the mirror image of the knot.
\item[--] (flype) means that we obtained the strongly quasipositive representative by applying a flype to the known braid representative. 
\item[--] (Theorem 3.1)/(Theorem 3.3) means that the strongly quasipositive braid representative is obtained by applying Theorem 3.1/Theorem 3.3.
\item[--] (SQP-unknown) means that the strongly quasipositivity was previously unknown, according to KnotInfo as of November 2016. There are thirteen such knots. 
\item[--] (QP-known) means that the quasipositivity was previously known, according to KnotInfo as of November 2016. 
\item[--] (positive braid) means that the knot is a positive braid knot. 
\end{itemize}

\subsection{Tables of strongly quasipositive knots}\label{subsec-Table-SQP}

The following is the list of all the knots up to 12 crossings for which the Bennequin inequality is an equality; equivalently, $\delta_3(\K)=0$. 
We confirmed that all of them are strongly quasipositive, and we list their strongly quasipositive braid representatives. Consequently, the following is a complete list of strongly quasipositive knots up to 12 crossings.

\begin{tabular}{|l || l | l |}
\hline
Knot & Strongly quasipositive braid representative   & Comment
\\ \hline
$3_1$  & 1,1,1 &  positive braid
\\ \hline
$5_1$  & 1,1,1,1,1 &  positive braid
\\ \hline
$5_2$  & 1,1,2 (2,1,-2) &
\\ \hline
$7_1$  & 1,1,1,1,1,1,1 &  positive braid
\\ \hline
$7_2$ &  1,1,(3,2,-3), (2,1,-2), 3 &
\\ \hline
$7_3$ & 1,1,1,1,2, (2,1,-2) &  Theorem~\ref{1negative}
\\ \hline
$7_4$ & 1, (3,2,-3), (3,2,1,-2,-3), (2,1,-2),3 & 
\\ \hline
$7_5$ & 1, 1,1,2, (2,1,-2), (2,1,-2) & Theorem~\ref{1negative}
\\ \hline
$8_{15}$ & 1, (3,2,-3), (3,2,-3), (3,2,1,-2,-3), (3,2,1,-2,-3),2,3 & 
\\ \hline
$9_1$ & 1,1,1,1,1,1,1,1,1 & positive braid
\\ \hline
$9_2$ & 2, (2,1,-2), (4,3,2,-3,-4), 2, (3,2,1,-2,-3), 4 & 
\\ \hline
$9_3$ & 1,1,1,1,1,1,2, (-1,2,1) &  Theorem~\ref{1negative}
\\ \hline
$9_4$ & 1,1,1,1, (3,2,-3), (2,1,-2), 3 & 
\\ \hline
$9_5$ & 2, (2,1,-2), (2,1,-2), (4,3,2,-3,-4), (3,2,1,-2,-3), 4 & 
\\ \hline
$9_6$ & 1,1,1,1,1, 2, (2,1,-2), (2,1,-2) &  Theorem~\ref{1negative}
\\ \hline
$9_7$ & 1,1,1,(3,2,-3), (2,1,-2), 3,3 & 
\\ \hline
$9_9$ & 1,1,1,1,2,(2,1,-2), (2,1,-2), (2,1,-2) &  Theorem~\ref{1negative}
\\ \hline
$9_{10}$ & 1, (3,2,-3), (3,2,1,-2,-3),(3,2,1,-2,-3),(3,2,1,-2,-3),(2,1,-2),3 & 
\\ \hline
$9_{13}$ & 1,1,1,(3,2,-3),(3,2,1,-2,-3),(2,1,-2),3 & 
\\ \hline
$9_{16}$ & 1,1,1,2,2,(2,1,-2),(2,1,-2),(2,1,-2) &  Theorem~\ref{1negative}
\\ \hline
$9_{18}$ & 1,1,(3,2,-3),(3,2,1,-2,-3),(3,2,1,-2,-3),(2,1,-2),3 & 
\\ \hline
$9_{23}$ & 1,1,(3,2,-3),(3,2,1,-2,-3),(2,1,-2),3,3 & 
\\ \hline
$9_{35}$ & 2, (2,1,-2), (3,2,-3), (2,1,-2), (4,3,2,-3,-4),(4,3,2,1,-2,-3,-4) & 
\\ \hline
$9_{38}$ & 1, (3,2,-3), (3,2,-3), 2, (2,1,-2),3,(3,2,-3) & 
\\ \hline
$9_{49}$ &1, (3,2,-3),1,1,2,(2,1,-2),3 & 
\\ \hline
%
%
$10_{49}$ & (2,1,-2),(2,1,-2),(2,1,-2),(2,1,-2),1,(3,2,-3),(3,2,-3),2,3 & 
\\ \hline
$10_{53}$ & 1,2,(3,2,1,-2,-3),(2,1,-2),(4,3,-4),(4,3,-4),3,4  & 
\\ \hline
$10_{55}$ & 1,1,(3,2,-3),(2,1,-2),(4,3,-4),3,4 & 
\\ \hline
$10_{63}$ & 1,1,(4,3,2,1,-2,-3,-4),2,(2,1,-2),(2,1,-2),3,4 & 
\\ \hline
$10_{66}$ & 1,1,1,(3,2,1,-2,-3), 2,(2,1,-2),(2,1,-2),3,3 & 
\\ \hline
$10_{80}$ & (2,1,-2), (2,1,-2), (2,1,-2),1,1,(3,2,-3),(3,2,-3),2,3 & 
\\ \hline
$10_{101}$ & 1,1,2,3,(3,2,1,-2,-3),(2,1,-2),4,(4,3,-4) & 
\\ \hline
$10_{120}$ & 1,(3,2,-3),(4,3,2,1,-2,-3,-4),2,(2,1,-2),(4,3,-4),3,4 & 
\\ \hline
$10_{124}$ & 1,1,1,1,1,2,1,1,1,2  & positive braid
\\ \hline
$10_{128}$ & 1,1,1,(3,2,-3),1,1,(3,2,-3),2,3  & 
\\ \hline
$10_{134}$ & 1,1,1,(3,2,-3),1,1,2,3,3 & 
\\ \hline
$10_{139}$ & 1,1,1,1,2,1,1,1,2,2  & positive braid
\\ \hline
$10_{142}$ & 1,1,1,(2,1,-2),1,1,1,2,3 & 
\\ \hline
$10_{145}$ & 1,2,(2,1,-2),3,2,(2,1,-2),3 & 
\\ \hline
$10_{152}$ & 1,1,1,2,2,1,1,2,2,2  & positive braid
\\ \hline
$10_{154}$ & 1,1,2,(-1,2,1),3,2,2,2,3 & 
\\ \hline
$10_{161}$ &1,1,1,2,(-1,2,1),1,2,2 & 
\\ \hline
\end{tabular}

\begin{tabular}{|l || l | l | }
\hline
Knot & Strongly quasipositive braid representative   & Comment 
\\ \hline
$11_{a43}$ & 1,1,(3,2,-3),(3,2,1,-2,-3),2,(4,3,-4),(4,3,-4),3,3,4 &
\\ \hline
$11_{a94}$ & 1,1,1,(3,2,-3),(3,2,-3),(3,2,1,-2,-3),(2,1,-2),3,3 &
\\ \hline
$11_{a95}$ & 1,1,(4,3,2,-3,-4),(4,3,2,1,-2,-3,-4),3,(3,2,1,-2,-3),(4,3,-4),3 &
\\ \hline
$11_{a123}$ & 2,(2,1,-2),(4,3,2,-3,-4),(4,3,2,1,-2,-3,-4),(3,2,-3),2,(3,2,1,-2,-3),4 &
\\ \hline
$11_{a124}$ & 1,1, (3,2,-3),(3,2,-3),2,(2,1,-2),3,3,(3,2,1,-2-3) &
\\ \hline
$11_{a186}$ & 1,1,(3,2,-3),(3,2,1,-2,-3),(3,2,1,-2,-3),(3,2,1,-2,-3),(2,1,-2),3,3 &
\\ \hline
$11_{a191}$ &1,1,1,1,(3,2,-3),(3,2,1,-2,-3),(2,1,-2),3,3 &
\\ \hline
$11_{a192}$ &2, (2,1,-2),1,(3,2,1,-2,-3),2,(4,3,-4),3,4 &
\\ \hline
$11_{a200}$ & (4,3,2,1,-2,-3,-4),(2,1,-2)$^2$ (4,3,2,-3,-4),2,3,4,4 & flype
\\ \hline
$11_{a227}$ &1,1,(3,2,-3)$^2$,2,(2,1,-2),3,(3,2,1,-2,-3)$^2$ &
\\ \hline
$11_{a234}$ & 1,1,1,1,1,1,1,2, (2,1,-2),(2,1,-2)  & Theorem~\ref{1negative}
\\ \hline
$11_{a235}$ &1,1,(3,2,-3),(3,2,1,-2,-3)$^4$,(2,1,-2),3 &
\\ \hline
$11_{a236}$ & 1,1,1,(3,2,-3),(3,2,1,-2,-3)$^2$,(2,1,-2),3,3&
\\ \hline
$11_{a237}$ & (4,3,2,1,-2,-3,-4), (2,1,-2)$^2$, (4,3,-4), (4,3,2,-3,-4), 2,3,4 & flype
\\ \hline
$11_{a238}$ &2,2,(2,1,-2),1,1,(4,3,2,-3,-4),(3,2,1,-2,-3),4 &
\\ \hline
$11_{a240}$ & 1,1,1,1,1,2,2,(2,1,-2),(2,1,-2),(2,1,-2) & Theorem~\ref{1negative}
\\ \hline
$11_{a241}$ &1,1,1,(3,2,-3)$^2$,(3,2,1,-2,-3)$^2$,(2,1,-2),3 &
\\ \hline
$11_{a242}$ &1,1,1,1,1,(3,2,-3),(2,1,-2),3,3 &
\\ \hline
$11_{a243}$ &1,(4,3,2,-3,-4),(3,2,1,-2,-3),4,4,(4,3,-4),3 &
\\ \hline
$11_{a244}$ &1,1,(3,2,-3)$^2$,2,(2,1,-2)$^2$,3,(3,2,1,-2,-3) &
\\ \hline
$11_{a245}$ & 1,1,(3,2,-3),(2,1,-2)$^2$,3,3,3&
\\ \hline
$11_{a246}$ & 1,1,1,(4,3,2,-3,-4),(3,2,1,-2,-3),(4,3,-4),3,3&
\\ \hline
$11_{a247}$ & 2,2,(2,1,-2),(5,4,3,2,-3,-4,-5),(5,4,3,2,1,-2,-3,-4,-5),(5,4,-5),4&
\\ \hline
$11_{a263}$ &1,1,1,2,2,3, (2,1,-2),(2,1,-2),(2,1,-2),3,3 & Theorem~\ref{1negative}
\\ \hline
$11_{a291}$ &1,(3,2,-3)$^4$,2,(2,1,-2),3,(3,2,1,-2,-3) &
\\ \hline
$11_{a292}$ & 1,(4,3,2,-3,-4),(4,3,2,1,-2,-3,-4),(3,2,-3),(3,2,1,-2,-3),4,(4,3,-4),3&
\\ \hline
$11_{a298}$ & 1,1,1,(3,2,-3)$^2$,2,(2,1,-2),3,(3,2,1,-2,-3)&
\\ \hline
$11_{a299}$ & 2,(2,1,-2),(4,3,2,-3,-4)$^2$,(3,2,-3),(3,2,1,-2,-3),4,(4,3,2,1,-2,-3,-4)&
\\ \hline
$11_{a318}$ & 1,(3,2,-3)$^2$,2,(2,1,-2)$^2$, 3, (3,2,-3)$^2$&
\\ \hline
$11_{a319}$ & 1,(3,2,-3)$^3$, 2, (2,1,-2),3,(3,2,1,-2,-3)$^2$&
\\ \hline
$11_{a320}$ & 1,1,(4,3,2,-3,-4),(3,2,1,-2,-3),4,2,(2,1,-2),(4,3,-4)&
\\ \hline
$11_{a329}$ & 1,(4,3,2,-3,-4),(4,3,2,1,-2-,-3,-4),(3,2,1,-2,-3),2,(2,1,-2),3,4&
\\ \hline
$11_{a334}$ & 1,1,1,1,1,2,(2,1,-2),(2,1,-2),(2,1,-2),(2,1,-2), & Theorem~\ref{1negative}
\\ \hline
$11_{a335}$ & 1,1,1,1,(3,2,-3),(3,2,1,-2,-3)$^2$,(2,1,-2),3&
\\ \hline
$11_{a336}$ & 1,1,1,1,1,(3,2,-3),(3,2,1,-2,-3),(2,1,-2),3&
\\ \hline
$11_{a337}$ & 1,(2,1,-2),1,1,(4,3,2,-3,-4),(4,3,2,1,-2,-3,-4),(3,2,1,-2,-3),4&
\\ \hline
$11_{a338}$ &1,1,1,1,2,2,2,(2,1,-2),(2,1,-2),(2,1,-2) & Theorem~\ref{1negative}
\\ \hline
$11_{a339}$ &1,1,1,1,(3,2,-3),(2,1,-2),3,3,3 &
\\ \hline
$11_{a340}$ &1,1,1,(3,2,-3)$^3$,(3,2,1,-2,-3),(2,1,-2),3 &
\\ \hline
$11_{a341}$ & 1,1,1,(2,1,-2),(4,3,2,-3,-4),(4,3,2,1,-2,-3,-4),(3,2,1,-2,-3),4&
\\ \hline
\end{tabular}

\begin{tabular}{|l || l | l | }
\hline
Knot & Strongly quasipositive braid representative   & Comment 
\\ \hline
$11_{a342}$ & 1,1,1,1,(2,1,-2),(4,3,2,-3,-4),(3,2,1,-2,-3),4&
\\ \hline
$11_{a343}$ & 2,(2,1,-2),1,(5,4,3,2,-3,-4,-5),(4,3,2,1,-2,-3,-4),(5,4,-5),4&
\\ \hline
$11_{a353}$ & 1,(3,2,-3)$^3$, 2, (2,1,-2)$^2$,3, (3,2,1,-2,-3) & Theorem~\ref{separating}
\\ \hline
$11_{a354}$ &1,1,(4,3,2,-3,-4),(4,3,2,1,-2,-3,-4),(3,2,-3),(3,2,1,-2,-3),(4,3,-4),3 &
\\ \hline
$11_{a355}$ & 1,1,1,1,1,1,2,(2,1,-2),(2,1,-2),(2,1,-2) &  Theorem~\ref{1negative}
\\ \hline
$11_{a356}$ & 1,1,1,(3,2,-3),(3,2,1,-2,-3)$^3$ (2,1,-2),3 & 
\\ \hline
$11_{a357}$ & 1,1,1,(3,2,-3),(3,2,1,-2,-3),(2,1,-2),3,3,3  &
\\ \hline
$11_{a358}$ & 1,1,1,1,1,1,(3,2,-3),(2,1,-2),3 &
\\ \hline
$11_{a359}$ & 2, (2,1,-2),1,1,1,(4,3,2,-3,-4), 3,4  & 
\\ \hline
$11_{a360}$ & 2, (2,1,-2),1,1,(4,3,2,-3,-4),2,2,3,4 &
\\ \hline
$11_{a361}$ & 1,1,1,(4,3,2,-3,-4),(3,2,1,-2,-3),2,3,4 &
\\ \hline
$11_{a362}$ & 1,(3,2,-3),(2,1,-2),3,(5,4,-5),(4,3,2,1,-2,-3,-4),5 & flype
\\ \hline
$11_{a363}$ &1,(2,1,-2),(5,4,3,2,-3-4,-5),(5,4,3,2,1,-2,-3,-4,-5),&\\
&(4,3,2,1,-2,-3,-4),(5,4,-5),4 &
\\ \hline
$11_{a364}$ & 1,1,1,1,1,1,1,1,2,(2,1,-2)  & Theorem~\ref{1negative}
\\ \hline
$11_{a365}$ &  1, (3,2,-3),(3,2,1,-2,-3)$^5$, (2,1,-2),3 & 
\\ \hline
$11_{a366}$ & (4,3,2,1,-2,-3,-4),(2,1,-2),(4,3,-4)$^2$, (4,3,2,-3,-4), 2,3,4  & flype
\\ \hline
$11_{a367}$ & 1,1,1,1,1,1,1,1,1,1,1  & positive braid
\\ \hline
\end{tabular}

 {\vspace{1in}\hspace{1cm}
\begin{tabular}{|l || l | l | }
\hline
Knot & Strongly quasipositive braid representative   & Comment 
\\ \hline
$11_{n77}$ & 1,1,2,2,1,3,2,2,2,3,3 & positive braid 
\\ \hline
$11_{n93}$ & 1,1, (3,2,-3),1,1,2,(2,1,-2),3,3 & Theorem \ref{separating}
\\ \hline
$11_{n126}$ & 1,1,(3,2,-3),1,1,(3,2,-3),2,(2,1,-2),3 & Theorem \ref{separating}
\\ \hline
$11_{n136}$ & 1,1,2,3,2,(3,2,1,-2,-3),2,(2,1,-2),(2,1,-2) & 
\\ \hline
$11_{n169}$ & 1,1,1,(3,2,-3),1,1,2,(2,1,-2),3   & Theorem \ref{separating}
\\ \hline
$11_{n171}$ & 2, (2,1 -2), (4,3,2,-3,-4),1,1,(3,2,-3),2,(3,2,1,-2,-3),4 &
\\ \hline
$11_{n180}$ & 1,1,1,2,3,2,(3,2,1,-2,-3),2,(2,1,-2) & \\ \hline
$11_{n181}$ & 1, (4,3,2,-3,-4),1,1,3,(3,2,-3),(3,2,1,-2,-3),(4,3,-4) &
\\ \hline
$11_{n183}$ & 1,1,(3,2,1,-2,-3),2,(2,1,-2),3,2,2,3  & Theorem~\ref{1negative}
\\ \hline
\end{tabular}
}

\begin{tabular}{|l || l | l| l |}
\hline
Knot & Strongly quasipositive braid representative   & Comment
\\ \hline
$12_{a43}$  &(2,1,-2)$^2$,(4,3,-4),3,4,1,(3,2,-3)$^2$,2,4& 
\\ \hline
$12_{a52}$ &1,2,2,2,3,(2,1,-2)$^4$,3,1 &
\\ \hline
$12_{a53}$ & 1,2,(2,1,-2),(4,3,-4)$^2$,1,3,4,(4,3,2,-3,-4)$^2$&
\\ \hline
$12_{a55}$ & 1,2,(2,1,-2),3,3,3,4,(4,3,2,-3,-4)$^2$&
\\ \hline
$12_{a56}$ & 1,(3,2,-3),(5,4,-5),4,5,(4,3,-4),5,(2,1,-2)$^2$&
\\ \hline
$12_{a82}$ & (2,1,-2)$^2$,1,2,(4,3,-4),3,4,(3,2,-3),2,4&
\\ \hline
$12_{a93}$ & 2,2,2,2,(2,1,-2),3,3,1,1,(3,2,-3),1&
\\ \hline
$12_{a94}$ & 1,2,(2,1,-2),3,3,3,1,(3,2,-3),4,(4,3,2,-3,-4)&
\\ \hline
$12_{a96}$ & 1,2,(2,1,-2),3,4,1,(4,3,2,-3,-4)$^4$&
\\ \hline
$12_{a97}$ & (3,2,-3),2,3,(5,4,-5),4,4,(5,4,3,2,1,-2,-3,-4,-5),5,(2,1,-2)&
\\ \hline
$12_{a102}$ & 1,2,(2,1,-2),(4,3,-4)$^2$,1,3,(3,2,-3),4,(4,3,2,-3,-4)&
\\ \hline
$12_{a107}$ & 1,2,(2,1,-2),(4,3,-4),3,3,1,(3,2,-3),4,(4,3,2,-3,-4)&
\\ \hline
$12_{a143}$ & 1,1,1,1,1,1,(2,1,-2),3,3,(3,2,-3),1&
\\ \hline
$12_{a144}$ & 1,2,(2,1,-2),3,1,(3,2,-3)$^3$,4,(4,3,2,-3,-4)&
\\ \hline
$12_{a145}$ &1,2,(2,1,-2),3,1,(3,2,-3),4,4,4,(4,3,2,-3,-4) &
\\ \hline
$12_{a152}$ & 1,2,3,(5,4,-5),4,4,(5,4,3,-4,-5),5,3,(2,1,-2),3&
\\ \hline
$12_{a156}$ & 1,(4,3,2,-3,-4),(2,1,-2),(4,3,-4),(3,2,1,-2,-3),(5,4,-5),4,4,5&
\\ \hline
$12_{a293}$ & (3,2,1,-2,-3)$^2$,4,1,(4,3,2,-3,-4)$^2$,2,(4,3,-4),3,4&  
\\ \hline
$12_{a295}$ & (3,2,1,-2,-3)$^2$,1,(3,2,-3),4,2,(4,3,-4),3,4&
\\ \hline
$12_{a319}$ & (2,1,-2)$^2$,3,1,(3,2,-3),4,2,2,2,(4,3,-4)&
\\ \hline
$12_{a320}$ & 1,2,(2,1,-2),(4,3,-4),(3,2,1,-2,-3),(5,4,-5),3,3,3,4,5&
\\ \hline
$12_{a344}$ & 1,2,(4,3,-4),3,3,3,(4,3,2,-3,-4)$^2$,4,4&
\\ \hline
$12_{a345}$ & 1,2,(5,4,3,-4,-5),3,4,5,(4,3,2,1,-2,-3,-4)$^2$,5&
\\ \hline
$12_{a355}$ & (2,1,-2),1,(4,3,-4)$^2$,3,3,(4,3,2,-3,-4),4,4,4,4&
\\ \hline
$12_{a356}$ & 1,(3,2,-3),4,(2,1,-2)$^2$,(5,4,3,-4,-5)$^2$,(4,3,-4),(5,4,-5)&
\\ \hline
$12_{a367}$ & 2,2,2,2,(2,1,-2)$^3$,3,3,(3,2,-3),1&
\\ \hline
$12_{a368}$ &1,2,(2,1,-2)$^3$,3,1,(3,2,-3),4,(4,3,2,-3,-4) &
\\ \hline
$12_{a391}$ &(2,1,-2)$^2$,3,3,1,3,(3,2,-3),4,2,(4,3,-4) &
\\ \hline
$12_{a392}$ &1,(4,3,2,-3,-4),(3,2,1,-2,-3),(2,1,-2),(5,4,3,-4,-5),3,3,4,5 &
\\ \hline
$12_{a420}$ &1,1,2,(4,3,-4),3,4,(3,2,1,-2,-3)$^2$,4,4,4 &
\\ \hline
$12_{a421}$ &1,2,(5,4,3,-4,-5),3,(5,4,-5),(3,2,1,-2,-3)$^2$,4,5 & 
\\ \hline
$12_{a431}$ &1,(4,3,2,-3,-4),1,2,(2,1,-2),(4,3,-4),3,(3,2,1,-2,-3)$^2$,4 &
\\ \hline
$12_{a432}$ &(2,1,-2)$^2$,3,3,1,(3,2,-3),4,2,2,(4,3,-4) &
\\ \hline
$12_{a442}$ &(3,2,-3),4,1,2,2,(2,1,-2)$^4$,(4,3,-4) &
\\ \hline
$12_{a443}$ & 2,(4,3,-4),5,(2,1,-2),(3,2,-3),(3,2,1,-2,-3)$^2$,(5,4,-5),4&
\\ \hline
$12_{a490}$ &(2,1,-2)$^2$, (4,3,-4), 3,1, (3,2,-3),2,(4,3,-4) &
\\ \hline
$12_{a574}$ &2, 2, (2,1,-2)$^5$, 3, 3, (3, 2, -3), 1 &
\\ \hline
$12_{a575}$ & (3,2,-3),1,2,2,(2,1,-2)$^2$, 3,3,4,(4,3,-4)&
\\ \hline
$12_{a586}$ & (2,1,-2)$^2$, 3, 1, (3,2,-3), 4, 2, (4, 3, -4)&
\\ \hline
$12_{a610}$ & 1,1,2,3,(3,2,1,-2,-3),(2,1,-2),(5,4,-5),(4,3,-4),5&
\\ \hline
$12_{a615}$ & (2,1,-2)$^2$, 3, 1, (3,2,-3)$^2$, 4, 2, (4,3,-4)$^2$&
\\ \hline
$12_{a647}$ & 1, 2, 2, (2,1,-2)$^2$, 3, 1, 1, (3,2,-3)$^2$&
\\ \hline
$12_{a648}$ & (2,1,-2), (4, 3, -4)$^2$, 1, 3, 3, (4,3,2,-3,-4), 4, (3,2, -3)&
\\ \hline
$12_{a653}$ & 1, 1, (5,4,3,2,-3,-4,-5), (2,1,-2), 3, (5,4,-5), 4, (5,4,3,-4,-5), (4,3,-4) &
\\ \hline
$12_{a659}$ & (2,1,-2), 2, (4,3,-4), 2, 2, (4,3,-4), 3, (3,2,-3), 4, (4,3,2,-3,-4)&
\\ \hline
\end{tabular}

\begin{tabular}{|l||l|l|l|}
\hline
Knot & Strongly quasipositive braid representative   & Comment 
\\ \hline 
$12_{a679}$ & 1,2,(4,3,-4),5, (3,2,1,-2,-3), (5, 4, -5)$^2$, 4, 5&
\\ \hline
$12_{a811}$ & 2, 2, (2,1,-2), 3, 3, 1, 1, (3,2,-3), 1, (3,2,1,-2,-3)$^2$&
\\ \hline
$12_{a813}$ & 2, 2, 2, 2, (2,1,-2), 3, 3, (3,2,-3), 1, (3,2,1,-2,-3)$^2$&
\\ \hline
$12_{a814}$ & (2,1,-2), (4,3,-4), 3, 3, 3, (4,3,2,1,-2,-3,-4), 4, 4, 1, 2&
\\ \hline
$12_{a817}$ & 2, 2, 2, (2,1,-2), 3, 3, 1, (3,2,-3), 1, (3,2,1,-2,-3)$^2$&
\\ \hline
$12_{a828}$ & (4,3,2,-3,-4),1,(4,3,2,1,-2,-3,-4),2,(2,1,-2),3,(3,2,1,-2,-3), 4,4,4&
\\ \hline
$12_{a876}$ & 2,2,2,2,(2,1,-2),3, 3, 3, 3, (3 2 -3), 1&
\\ \hline
$12_{a877}$ & (2,1,-2), (4,3,-4), 3, 4, (3,2,1,-2,-3), 4, 1, 2, 4, 4&
\\ \hline
$12_{a880}$ & 1, (3,2,-3), 4,5, (2,1,-2), (5,4,3,-4,-5), 3, &\\
& (5,4,3,2,1,-2,-3,-4,-5), (5,4,-5)&
\\ \hline
$12_{a900}$ & (3,2,1,-2,-3)$^2$, (2,1,-2)$^2$, (4,3,-4), 1, (4 3 2 -3,-4), 2, 3, 4&
\\ \hline
$12_{a973}$ & (4,3,2,-3,-4), 1, 2, (2,1 -2), (4,3,-4), 3, (3,2,1,-2,-3), 4, 4, 4&
\\ \hline
$12_{a974}$ & 2,3,(5,4,-5),(3,2,1,-2,-3),1,(5,4,3,2,-3,-4,-5), &\\
&(4,3,2,-3,-4),(4,3,2,1,-2,-3,-4),5& flype
\\ \hline
$12_{a995}$ & (4,3,2,-3,-4), 1, 2, (2,1,-2)$^3$, (4,3,-4), 3, (3,2,1,-2,-3), 4&
\\ \hline
$12_{a996}$ & (2,1,-2), 2, (4,3,-4), (5,4,3,2,-3,-4,-5), 3, (3,2,-3), (5,4,-5), 4, 5& flype
\\ \hline
$12_{a1004}$ & (3,2,-3), 1, (3,2,1,-2,-3), 4, 4, 2, (2,1,-2), &\\
& (4,3,-4), (4,3,2,1,-2,-3,-4), 3&
\\ \hline
$12_{a1035}$ & (3,2,1,-2,-3), (2,1,-2), 1, (4,3,2,-3,-4), 2, 2, 2, 2, 3, 4&
\\ \hline
$12_{a1037}$ & (5,4,3,2,1,-2,-3,-4,-5),1,(5,4,3,2,-3,-4,-5),2,3, &\\
&(5,4,-5),4,(5,4,3,-4,-5),(4,3,-4)& flype
\\ \hline
$12_{a1097}$ & (3,2,1,-2,-3),1,(3,2,-3),(5,4,-5),(5,4,3,2,1,-2,-3,-4,-5),2,(2,1,-2),3,4& flype
\\ \hline
$12_{a1112}$ & (3,2,1,-2,-3), (2,1,-2), (4,3,-4)$^3$, 1, (4,3,2,-3,-4), 2, 3, 4&
\\ \hline
$12_{a1113}$ & (4,3,2,1,-2,-3,-4),1,(3,2,-3),2,(5,4,-5),(5,4,3,-4,-5),3,4,5&flype
\\ \hline
\end{tabular}

\begin{tabular}{|l || l | l| l |}
\hline
Knot & Strongly quasipositive braid representative   & Comment 
\\ \hline
$12_{n74}$ & 1,1,1,2,2,(,2,1,-2),3,(2,1,-2)$^2$, 3, 3&
\\ \hline
$12_{n77}$ & 1,2, 3, 4, 3, 3, 3, (4,3,2,1,-2,-3,-4), 4, (3,2,1,-2,-3)&
\\ \hline
$12_{n88}$ & 1,2,2,(2,1,-2), 3, (3,2,-3), (3,2,1,-2,-3), (2,1,-2), 1&
\\ \hline
$12_{n91}$ & 1, 2, 2, 1, 1, 1, (3,2,-3), 2, 1, 3, 3&
\\ \hline
$12_{n96}$ & 1, 2, 3, 3, 4, (2,1,-2)$^2$, (4,3,-4), 3, 4&
\\ \hline
$12_{n100}$ & 2, 2, 3, 4, (3,2,-3)$^2$, (4,3,2,1,-2,-3,-4),& \\
& (3,2,-3), (3,2,1,-2,-3), (2,1,-2)&
\\ \hline
$12_{n105}$ & 1, 1, 1, 2, (2,1,-2), 3, 2, 2, 1, 3, 3&
\\ \hline
$12_{n110}$ & 1, 2, (2,1,-2), 3, 4, (4,3,-4), 1, 2, 2, 3&
\\ \hline
$12_{n133}$ & (2,1,-2), 1, 1, 2, 2, 3, (3,2,-3), (3,2,1,-2,-3)$^2$&
\\ \hline
$12_{n136}$ & 1, 1, 2, 2, 1, 1, (3,2,-3), 2, 1, 3, 3&
\\ \hline
$12_{n148}$ &  
1, 2, (2,1,-2), 3, 2, (2,1,-2)$^3$, 3$^2$ & SQP-unknown, mirror
\\ \hline
$12_{n149}$ & 
1,2,(4,3,-4),3,(2,1,-2),(4,3,-4),(4,3,2,1,-2,-3,-4),3 & SQP-unknown, mirror 
\\ \hline
$12_{n153}$ & 1, (3,2,-3), 1, (3,2,-3), 2, 2, 3, 1, (3,2,-3), 1, (3,2,-3)&
\\ \hline
$12_{n166}$ & 1, 2, 2, 2, 2, (2,1,-2), 3, (2,1,-2), 3, 3&
\\ \hline
$12_{n169}$ & 1, 2, (2,1,-2), 1, 2, 2, 3, (3,2,-3), 4, (4,3,2,-3,-4)&
\\ \hline
$12_{n177}$ & 1, 2, 3, 4, (4,3,-4), (2,1,-2)$^2$, (4,3,-4)$^2$, 3&
\\ \hline
$12_{n203}$ & 1, 2, 2, 3, 1, 1, (3,2,-3), 4, 2, (4, 3, -4)&
\\ \hline
$12_{n217}$ & 1,2,(2,1,-2),$^2$, 3,4,2,2,2,(4,3,-4)&
\\ \hline
$12_{n242}$ & 1,2,2,1,1,2,2,2,2,2,2,2& positive braid
\\ \hline
$12_{n243}$ & 1,2,2,3,3,2,2,2,2,(2,1,-2),3&
\\ \hline
$12_{n244}$ & 1,1,1,2,2,3,3,2,2,(2,1,-2),3&
\\ \hline
$12_{n245}$ & (2,1,-2), 3, 4, (2,1,-2), (4,3,-4), (2,1,-2),3 , 3, 1, 2&
\\ \hline
$12_{n251}$ & 1, (4,3,2,-3,-4)$^2$,1 ,1, (4,3,2,-3,-4)$^2$,2,3,4&
\\ \hline
$12_{n259}$ & (2,1,-2),2,2,3,1,(3,2,-3), (3,2,1,-2,-3)$^2$,(2,1,-2)&
\\ \hline
$12_{n276}$ & 2,2,1,2,2,3,1,3,2,(2,1,-2),3& mirror 
\\ \hline
$12_{n289}$ & 1,2,3,3,4,(2,1,-2)$^2$, (4,3,-4)$^2$,3&
\\ \hline
$12_{n292}$ & 1,(3,2,-3)$^2$, 1,1,1,(3,2,-3),2,2,1,3&
\\ \hline
$12_{n293}$ & 
(2, 1, -2), 3, (3,2,-3), (2,1,-2), 3, 2, 2 & SQP-unknown
\\ \hline
$12_{n305}$ & 1,2,3,3,(2,1,-2)$^2$,3,3,3,3,3&
\\ \hline
$12_{n308}$ & 1,2,(4,3,-4)$^2$, (2,1,-2)$^2$, (4,3,-4)$^2$, 3, 4&
\\ \hline
$12_{n321}$ & 
(2, 1, -2), (3, 2, -3)$^2$,(2,1,-2),3 , 2, 3  &SQP-unknown 
\\ \hline
$12_{n328}$ &1,2,3,1,3,2,2,(2,1,-2),3 ,(2,1,-2),2&
\\ \hline
$12_{n329}$ &  1, 2, 1, (4,3,-4), 3, (2,1,-2)$^2$,(4,3,4), 3 & mirror 
\\ \hline
$12_{n332}$ & 
1,2,(4,3,-4),(3,2,1,-2,-3)$^2$,4,3,3& SQP-unknown, mirror 
\\ \hline
$12_{n338}$ & 1,2,3,3,3,3,(2,1,-2)$^2$, 3, 3, 3&
\\ \hline
$12_{n341}$ & 1,2,3,4,(4,3,-4), (2,1,-2)$^2$, 3, 3, 3&
\\ \hline
$12_{n366}$ & 1,(3,2,3$^{-1}$)$^2$,2,3,1,3,(3,2,3$^{-1}$),2 & mirror 
\\ \hline
\end{tabular}

\begin{tabular}{|l || l | l| l |}
\hline
Knot & Strongly quasipositive braid representative   & Comment 
\\ \hline
$12_{n374}$ & 1,2,3,3,(2,1,-2)$^2$, 3,3,(2,1,-2)$^2$, 3&
\\ \hline
$12_{n386}$ & 1,(3,2,-3)$^2$, 1, 1, 2, 2, 3, (3,2,-3), 1, (3,2,-3)&
\\ \hline
$12_{n402}$ & 1,1,1,(3,2,-3),2,1,(3,2,-3),2,1& mirror
\\ \hline
$12_{n404}$ &  
(3,2,1,-2,-3), 4,3,3, (2,1,-2), 1, 2, (4,3,-4)  & SQP-unknown, mirror 
\\ \hline
$12_{n406}$ & 1,2,(4,3,-4)$^2$, 3, 4, (3,2,1,-2,-3), (2,1,-2), (4,3,-4)$^2$&
\\ \hline
$12_{n417}$ & (2,1,-2)$^2$,1,2,2,2,1,1,2,2&
\\ \hline
$12_{n426}$ & 1,2,3,1,3,2,(2,1,-2),3,2,1,2&
\\ \hline
$12_{n432}$ &
(3,2,1,-2,-3), 4, (2,1,-2), 3,3, 1, 2, (4,3,-4) & SQP-unknown, mirror
\\ \hline
$12_{n453}$ & 1,2,2,(4,3,-4),(4,3,2,-3,-4),2,3,4,(2,1,-2)$^2$&
\\ \hline
$12_{n472}$ & 1,2,2,2,2,1,1,2,2,2,2,2& positive braid
\\ \hline
$12_{n473}$ & 1,2,2,2,2,3,3,2,2,(2,1,-2),3&
\\ \hline
$12_{n474}$ & 1,2,3,3,2,2,2,2,2,3,(2,1,-2)&
\\ \hline
$12_{n477}$ & (2,1,-2),3,4,(2,1,-2)$^2$ (4,3,-4),(2,1,-2),(4,3,-4), 1, 2&
\\ \hline
$12_{n502}$ & 1,2,3,1,1,1,1,(3,2,-3),2,1,2&
\\ \hline
$12_{n503}$ & (2,1,-2),(4,3,-4),(2,1,-2),(4,3,-4),(2,1,-2),&\\
&(4,3,-4),3,4,1,2&
\\ \hline
$12_{n518}$ & (2,1,-2)$^2$, 3, 3, 2, 2, 2, 3, 1, 2, 2&
\\ \hline
$12_{n528}$ &
2,2, (2,1,-2), 3, 1, 2, (2,1,-2), 3,3  &SQP-unknown, mirror 
\\ \hline
$12_{n574}$ & 1,2,2,2,2,2,2,1,1,2,2,2& positive braid
\\ \hline
$12_{n575}$ & 1,2,3,3,2,2,2,3,2,2,(2,1,-2)&
\\ \hline
$12_{n576}$ & (2,1,-2), 3, 1, 3, 3, 2, 2, 2, 3, 3, 3&
\\ \hline
$12_{n581}$ & 1, 2, 3, 4, (4,3,2,-3,-4), 1, 1, 2, 2, 2&
\\ \hline
$12_{n585}$ & 1,2,3,4,(2,1,-2),(4,3,-4),(2,1,-2),(4,3,-4),(2,1,-2),3&
\\ \hline
$12_{n591}$ & 1,2,3,3,2,2,1,1,2,3,(2,1,-2)&
\\ \hline
$12_{n594}$ & 1,2,3,(3,2,-3)$^2$, 1, 1, (3,2,-3), (2,1,-2)&
\\ \hline
$12_{n600}$ & (2,1,-2),(4,3,-4)$^2$, 1, 2, 2, 3, 4, 2,(4,3,-4)&
\\ \hline
$12_{n638}$ & 1,(3,2,-3)$^2$, 1,1,2,2,3,(3,2,1,-2,-3)&
\\ \hline
$12_{n640}$ & (2,1,-2)$^2$, 1, 2, 2, 2, 2,1,1,2&
\\ \hline
$12_{n642}$ & (2,1,-2), (4,3,-4), 3, 2, (2,1,-2), (4,3,-4), 3, 2 & SQP-unknown, mirror
\\ \hline
$12_{n644}$ & 1, 2, 2, 2, 3, 1, 1,(3,2,-3), (2,1,-2)&
\\ \hline
$12_{n647}$ & (2,1,-2)$^2$, 1, 2, 2, 1, 1, 2, 2, 2&
\\ \hline
$12_{n655}$ & 1,2,3,3,(2,1,-2)$^2$, 3, (3,2,1,-2,-3)$^2$&
\\ \hline
$12_{n660}$ & 
1, (3,2,-3)$^2$, 2, 1, 3, (3,2,-3) 2,2  & SQP-unknown, mirror 
\\ \hline
$12_{n679}$ & 1,1,1,2,2,1,1,2,2,2,2,2& positive braid
\\ \hline
$12_{n680}$ & 1,2,(2,1,-2)$^2$,3,3,(2,1,-2)$^2$, 3, 3, 3&
\\ \hline
$12_{n688}$ & 1,1,1,2,2,2,2,1,1,2,2,2& positive braid
\\ \hline
$12_{n689}$ & 1,2,2,2,3,1,1,1,(3,2,-3),2,1&
\\ \hline
$12_{n691}$ & 1,2,3,(3,2,-3),1,1,1,(3,2,-3),2,2,1&
\\ \hline
$12_{n692}$ & 1,2,3,(3,2,-3)$^2$, 1, 1, 1, (3, 2, -3),2,1&
\\ \hline
$12_{n694}$ & 1,2,2,3,1,1,2,2,3,3,(2,1,-2)&
\\ \hline
$12_{n725}$ & 1,2,2,1,1,1,1,2,2,2,2,2& positive braid
\\ \hline
$12_{n750}$ & (2,1,-2),1,2,2,(2,1,-2),1,2,2&
\\ \hline
\end{tabular}

\begin{tabular}{|l || l | l| l |}
\hline
Knot & Strongly quasipositive braid representative   & Comment 
\\ \hline
$12_{n758}$ & (3,2,1,-2,-3),(2,1,-2),(4,3,-4),1,2,2,(4,3,-4),&\\
&3,(4,3,2,-3,-4),4&
\\ \hline
$12_{n764}$ & 1,(3,2,-3)$^2$,1,1,2,(2,1,-2),3,(3,2,1,-2,-3)&
\\ \hline
$12_{n801}$ & 
1, (3,2,-3), (2,1,-2), 3, (3,2,-3), 2, 3  & SQP-unknown
\\ \hline
$12_{n806}$ & 1,(3,2,-3),(3,2,1,-2,-3)$^2$,2,(2,1,-2),3,(3,2,1,-2,-3)$^2$&
\\ \hline
$12_{n830}$ &  
 (2,1,-2), 1,1, 2,2, (2,1,-2), 1, 2  &SQP-unknown 
\\ \hline
$12_{n850}$ & (2,1,-2)$^4$,1,2,2,1,1,2&
\\ \hline
$12_{n851}$ & (2,1,-2),1,2,2,3,(3,2,-3),2,3,3&
\\ \hline
$12_{n881}$ & (2,1,-2),3,(3,2,1,-2,-3),4,1,(4,3,2,-3,-4),2,(4,3,-4)&
\\ \hline
$12_{n888}$ & 1,1,1,2,2,2,1,1,1,2,2,2& positive braid
\\ \hline
\end{tabular}
\subsection{Tables of quasipositive knots with $\delta_3=1$}\label{subsec-tableQP}

The following is the list of knots up to 12 crossings that have $\delta_3=1$. (As we remarked earlier, there are no alternating knots with $\delta_3=1$, up to 12 crossings). 
We confirmed that they are all quasipositive, and list their quasipositive braid representatives.
The list shows that the answer to Question \ref{question:d=1} is ``Yes'' for knots up to 12 crossings. Unlike alternating knots, this list of non-alternating knots contains many quasipositive knots that are not positive. 
\\

\begin{tabular}{|l|| l|l|}
\hline
Knot & Quasipositive braid representative   & Comment 
\\ \hline 
$8_{21}$ & (2,1,-2)$^3$,2 & QP known
\\ \hline 
$9_{45}$ & (-2,1,2)$^3$,3,(3,2,-3) & QP known, mirror
\\ \hline 
$10_{127}$ & 1,1,1,(-2,1,2),(2,1,-2)$^2$ & QP known, 
\\ \hline 
$10_{131}$ & 1,(1,1,-2,-3,2,4,-2,3,2,-1,-1),(-2,3,2),4,3,           (-2,-3,2,4,-2,3,2) & QP known
\\ \hline 
$10_{133}$ & (2,-4,-3,2,3,4,-2),(2,-3,-2,3,-1,-3,2,3,4,-3,-2,3,1,-3,2,3,-2), & \\
& (2,-3,-2,3,1,-3,2,3,-2),1,1,2 & QP known
\\ \hline 
$10_{149}$ & (2,1,-2)$^3$,(-2,1,2)$^2$,2 & QP known
\\ \hline 
$10_{157}$ & 1,(-2,1,2)$^2$,2,(2,1,-2)$^2$ & QP known, mirror 
\\ \hline 
$10_{165}$ & (3-2,1,2,-3),2,(2,1,-2),3,(3,2,1,-2,-3) & QP known, mirror
\\ \hline 
$11_{n 1}$ & 1,1,(-3,2,3),(3,2,1,-2,-3),4,(4,3,-4) &
\\ \hline 
$11_{n 10}$ & (-2,1,2)$^3$,(2,1,-2)$^2$,3,(3,2,-3)&
\\ \hline 
$11_{n 14} $ & (-2,1,2)$^4$, (2,1,-2), 3, (3,2,-3)&
\\ \hline 
$11_{n 17}$ & 1,(-3,2,3),(-3,2,1,-2,3), (3,2,1,-2,-3),4,(4,3,-4) &
\\ \hline 
$11_{n 35}$ &  1, (-2,1,2)$^2$,(3,2,1,-2,-3),(2,1,-2),3,3 &
\\ \hline 
$11_{n 43}$ & 1, (-2,1,2)$^2$, (3,2,1,-2,-3),(2,1,-2)$^2$,3 &
\\ \hline
$11_{n 59}$ & (2,1,-2)$^3$, 1, (3,2,-3),2,(2,3,-2) &
\\ \hline 
$11_{n 63}$ & (-2,1,2)$^2$, (2,1,-2), (4,3,-4), (3,2,-3),4 &
\\ \hline
$11_{n 72}$ & 1, (-2,1,2), (3,2,1,-2,-3), (2,1,-2)$^2$, 3$^2$ &
\\ \hline
$11_{n 84}$ & 1, (-2,1,3,2,-3,-1,2), 1,(2,1,-2),3 & mirror
\\ \hline
$11_{n 89}$ & 1, 1, (-2,1,2),(3,2,-3),(2,1,-2)$^2$,3  & mirror 
\\ \hline
\end{tabular}

{\hspace{1cm}
\begin{tabular}{|l || l | l| l |}
\hline
Knot & Quasipositive braid representative   & Comment 
\\ \hline 
$11_{n 91}$ & 1,(3,2,-3),(2,1,-2,4,3,-4,2,-1,-2), (4,3,-4),3,4  & mirror
\\ \hline
$11_{n 95}$ & 1,(1,2,-1),2,3,2,(-1,2,1),3 &
\\ \hline
$11_{n 99}$ & (3,2,-3),(3,2,1,-2,-3)$^2$, (2,1,-2,3,2,-1,-2),(2,1,-2,1,2,-1,-2) & mirror
\\ \hline
$11_{n 105}$ & (2,1,-2)$^2$, 1, (-2,3,2),2,(2,3,-2)$^2$ &
\\ \hline
$11_{n 108}$ & 1, (3,2,-3), (-2,1,2)$^2$, (2,1,-2)$^2$, 3 & mirror
\\ \hline
$11_{n 109}$&  1,1,(-2,1,2)$^2$,(3,2,-3), (2,1,-2),3 & mirror
\\ \hline
$11_{n 113}$ & 1, (3,2,-3),(2,1,-2),(2,1,-2,4,3,-4,2,-1,-2), 3,4  & mirror
\\ \hline 
$11_{n 118}$ & 1,1,1,(2,1,-2),(-3,2,3),1,(3,2,-3) &
\\ \hline 
$11_{n 122}$ & (2,1,-2)$^3$,(3,-2,1,2,-3),(-3,2,3) & mirror
\\ \hline
$11_{n 134}$ & (-2,1,2)$^2$,(2,1,-2),(3,-2,1,2,-3),(3,2,-3) & mirror
\\ \hline
$11_{n 139}$ & 1,(-3,-2,1,2,3),(-4,3,2,-3,4),(4,3,-4)  &
\\ \hline
$11_{n 144}$ & 1,(-2,1,2)$^2$,(3,2,-3),(2,1,-2),3,3 &
\\ \hline 
$11_{n 162}$ & (-2,1,2),(2,1,-2),(-3,2,3),(3,2,1,-2,-3),(-4,3,4),4  &
\\ \hline
$11_{n 174}$ & (2,1,-2)$^2$,1,3,(-2,1,2),(3,2,-3),(2,1,-2) &
\\ \hline
$11_{n 185}$ &  (2,1,-2),1,3,(-2,1,2),(3,2,-3),2,(-3,2,3)  &
\\ \hline
\end{tabular}
}

\pagebreak

\begin{tabular}{|l || l |l|}
\hline
Knot & Quasipositive braid representative   & Comment 
\\ \hline
$12_{n58}$ & (2,1,-2),1,(3,2,-3),4,4,(3,2,-3),(2,3,-2),(2,-4,3,4,-2) & mirror
\\ \hline
$12_{n72}$ & (-2,-2,1,2,2),(-3,2,3),(-3,4,3),(-3,4,3),3,4 &
\\ \hline
$12_{n76}$ & 1,1,1,(3,3,2,-3,-3),(3,3,2,-3,-3),(-1,2,1),3 &
\\ \hline
$12_{n79}$ & 1(-3,-2,1,2,3),(-4,3,4),(-4,-4,3,4,4),2,2 &
\\ \hline
$12_{n81}$  &(-1,3,2,1,-2,-3,1),(-1,3,2,1,-2,-3,1),1,(3,2,-3),2,2,3 &
\\ \hline
$12_{n114}$  & 1,(1,1,2,-1,-1),(1,1,2,-1,-1),2,2,2,2,2 & QP-known
\\ \hline
$12_{n117}$  & 1,(-2,1,2),2,2,(-3,-3,2,3,3,),(-3,-3,2,3,3,),3 &
\\ \hline
$12_{n121}$  & 1,2,(-1,2,3,-2,1),(2,2,3,-2,-2),(2,2,3,-2,-2) & mirror
\\ \hline
$12_{n123}$  & (2,1,-2),1,(3,2,-3),(3,2,-3),(2,-3,4,3,-2),(2,3,-2),4,4 &
\\ \hline
$12_{n128}$   & 1,(-2,1,2),(-2,-2,3,2,2),(-2,-2,3,2,2),4,4,3,(3,4,-3) &
\\ \hline
$12_{n146}$ & 1,1,(3,2,-3), (-3,-4,3,2,-3,4,3), (3,2,-3), (-4,3,4) & mirror
\\ \hline
$12_{n155}$   & (1,2,-1)$^2$, 2,3, (-2,-2,3,3,2,2),1 &
\\ \hline
$12_{n168}$  & (3,-2,1,2,3)$^4$, (-3,2,3), 1, 3 &
\\ \hline
$12_{n171}$  & (-2,1,2), 2, (-3,2,1,-2,3), 3, 4, (4,3,-4) &
\\ \hline
$12_{n176}$  & 1, (2,2,1,-2,2,1,-2,4,-3,2,3,-4,-2), (-4,3,4), (3,2,-3) &
\\ \hline
$12_{n183}$  & 1, (3,2,-3), (-2,1,2), 4,4,2, (4,3,-4), (-3,2,3) &
\\ \hline
$12_{n191}$ &  1, (1,1,2,2,2,2,-1,-1),2,2,2 & QP-known 
\\ \hline
$12_{n194}$  & (2,1,-2)$^3$,3,1, (3,-2,1,2,-3), (-2,1,2) &
\\ \hline
$12_{n209}$ & (2,1,-2),3,(3,2,1,-2,-3), 1, (3,-2,1,2,-3), (-2,1,2)$^2$ & 
\\ \hline
$12_{n212}$ & (1,2,2,1,-2,3,-2,1,2,-3,-1), (3,2,-3,2,1,3,-2,-3) & mirror
\\ \hline
$12_{n213}$ & (-3,2,1,-2,3), 3, (2,1,-2)$^2$, 1, (-2,1,2), (-3,-2,1,2,3) &
\\ \hline
$12_{n222}$ & 1, (-3,2,3), (-3,2,1,-2,3), (-4,3,4), 1, (3,2,-3), 2,4 &
\\ \hline
$12_{n234}$ & 1,(-2,-2,1,1,2,2),2,2,2,2,2 & QP-known 
\\ \hline
$12_{n237}$ & 1, (-2,-2,3,3,2,2), 2,2,(2,1,-2),3 &
\\ \hline
$12_{n240}$ & 1,1,1,(-2,-2,3,3,2,2),(2,1,-2),3 &
\\ \hline
$12_{n247}$  & (2,1,-2),(-2,1,2)$^2$, 3, 3,4 &
\\ \hline
$12_{n249}$ & 1,(-2,-2,1,1,2,2),(-3,2,3),3,(3,2,-3) &
\\ \hline
$12_{n254}$ & (-2,1,2),(-3,2,3),(2,1,-2)$^2$,(3,2,-3)$^2$,1 &
\\ \hline
$12_{n270}$ & (2,1,-2),1,(4,3,2,-3,-4),(3,2,-3),(4,-3,2,3,-4),(-3,2,3) & mirror 
\\ \hline
$12_{n290}$ & (-1,3,2,1,-2,-3,1),1,1,(3,2,-3,2,2,3,-2,-3),(3,2,1,-2,-3),2 &
\\ \hline
$12_{n303}$  & 1, (-2,1,2), (-3,-3,2,2,3,3), 3,3,3 &
\\ \hline
$12_{n306}$  & 1, (-2,1,2), (-3,-3,2,2,3,3),(3,4,-3),4 &
\\ \hline
$12_{n316}$ & 1, (3,-2,1,2,-3), 2,2,(-3,2,3)$^2$,3 &
\\ \hline
$12_{n373}$  & (-2,1,2)$^2$, 3, (3,2,-3), (3,2,1,-2,-3), 1 &
\\ \hline
$12_{n375}$  & 1, (-2,1,2), 3, (3,2,-3), 2, (-3,2,3), 3 &
\\ \hline
$12_{n381}$ & (2,1,-2), (-4,3,4), 1, (-4,3,2,-3,4), (-4,3,4), (4,-3,2,3,-4) &
\\ \hline
$12_{n383}$ & 1, (-4,3,2,-3,4), (-2,1,2)$^2$, (-3,2,3), (4,3,-4) &
\\ \hline
$12_{n407}$ & 1, (-2,1,2), 3, (3,2,-3), (-3,2,3), 3,3 &
\\ \hline
$12_{n429}$  & (1,1,-3,2,1,-2,3,-1,-1),(1,1,-4,3,4,-1,-1),(1,1,3,2,-3,-1,-1),4,2,(4,3,-4) & mirror
\\ \hline
$12_{n441}$  & (2,1,-2)$^2$, 1, (-3,-2,1,2,3),(-3,2,3),3,2 &
\\ \hline
\end{tabular}
\pagebreak

\begin{tabular}{|l||l|l|}
\hline
Knot & Quasipositive braid representative   & Comment
\\ \hline
$12n_{496}$ &1, (3,2,-3),4,1,2,(-3,2,3),(4,3,-4),(-3,2,3) &
\\ \hline 
$12n_{510}$ &1, (3,2,-3), (-2,1,2),(3,2,-3),4,(-3,2,3),(4,3,-4),(-3,2,3) & mirror
\\ \hline 
$12n_{513}$ &(2,1,-2),(-2,1,2),3,3,(3,-2,1,2,-3),(3,2,-3) &
\\ \hline 
$12n_{520}$ &1, (3,2,-3), (-2,1,2), (-4,3,2,-3,4), (-3,2,3), (4,3,-4)  & mirror 
\\ \hline 
$12n_{564}$ & (-2,1,2), (4,3,-4),3,4,(3,2,-3),4 & mirror
\\ \hline 
$12n_{589}$&(-2,1,2), 3,3,(3,2,-3),(3,2,1,-2,-3),(2,1,-2)$^2$ &
\\ \hline 
$12n_{626}$&(2,1,-2),3,4,1,(3,2,-3)$^2$, (4,-3,2,3,-4),(-3,2,3) & mirror
\\ \hline 
$12n_{671}$&(2,1,-2),(-2,1,2)$^3$,3,2,(2,1,-2),(-3,2,3)$^2$& 
\\ \hline 
$12n_{674}$&1,1,1,(-2,-2,1,2,2)$^2$,2,2,2 & QP-known 
\\ \hline 
$12n_{677}$ &(2,1,-2), 3, 1, 1, (3,-2,1,2,-3), (-2,1,2), 2&
\\ \hline 
$12n_{682}$ &(2,1,-2)$^2$, (-2,1,2)$^2$, 3,2,(2,1,-2),(-3,2,3)$^2$ &
\\ \hline 
$12n_{698}$&(-2,1,2), 3, (3,-2,1,2,-3), (3,2,-3), (2,1,-2)$^3$ &
\\ \hline 
$12n_{699}$ &1, (3,-2,1,2,-3), (-3,2,3), 4, (4, -3, 2, 3, -4), (4, 3, -4) & mirror \\ \hline 
$12n_{700}$ &2, (2, 1, -2), 3, 3, 3, (3,2,1,-2,-3), (3,-2,1,2,-3)  & mirror
\\ \hline 
$12n_{701}$ &(-4,-3,2,1,-2,3,4), (4,-3,2,1,-2,3,-4), (4,3,-4), 3, 1,2 & mirror 
\\ \hline 
$12n_{722}$ & (2,1,-2), (-2,1,2)$^4$, 2, 2, 2 & QP-known
\\ \hline 
$12n_{724}$ &(2,1,-2), 3, 1, 1, 1, (3,-2,1,2,-3), (-2, 1, 2) &
\\ \hline 
$12n_{726}$ &2, (-3,2,3), (3,2,-3), (-4,3,4), (-2,1,2), 4 &
\\ \hline 
$12n_{729}$ &(2,1,-2), 3, (2,1,-2), 3, (-2,1,2), (-3,-2,1,2,3), (-3,2,3) &
\\ \hline 
$12n_{734}$ &(-2,1,2),3,(3,-2,1,2,-3), (3,2,-3)$^3$, (2,1,-2) &
\\ \hline 
$12n_{735}$ &1, (-3,-2,1,2,3),3,4,(4,3,2,-3,-4), (4,-3,2,3,-4) &
\\ \hline 
$12n_{738}$  &(2,1,-2)$^2$, (-2,1,2)$^2$,3, (3,-2,1,2,-3), (3,2,-3) &
\\ \hline 
$12n_{749}$& (-2,1,2),2,1,1,(2,2,1,-2,-2)$^2$ &
\\ \hline 
$12n_{753}$ & (3,2,1,-2,-3)$^2$,(-2,1,2),2,2,3,(3,2,-3) &
\\ \hline 
$12n_{770}$ &(2,1,-2), 3, (3,2,1,-2,-3)$^2$,(3,-2,1,2,-3), (-2,1,2), 2 &
\\ \hline 
$12n_{796}$&(-2,1,2)$^3$, 3, (3,-2,1,2,-3), (3,2,-3), (2,1,-2) &
\\ \hline 
$12n_{797}$ &(4,3,2,1,3,-2,1,2,-3,-1,-2,-3,-4),3,4,1,(4,3,2,-3,-4),2 &
\\ \hline 
$12n_{807}$ & (-2,1,2),3,2,(3,2,1,-2,-3),(2,1,-2),3,3 &
\\ \hline 
$12n_{814}$ &(-3,2,3),(-3,2,1,-2,3),4,(3,2,-3),(3,2,1,-2,-3), & \\
& (4,3,2,1,3,-2,1,2,-3,-1,-2,-3,-4) &
\\ \hline 
$12n_{820}$&  (2,1,-2)$^2$, 1, (-2,1,2)$^2$, 2,2, 2 & QP-known, mirror
\\ \hline 
$12n_{823}$ &1, (-2,1,2), (-3,2,3)$^2$, 3, (3,2,-3)$^2$ & mirror
\\ \hline 
$12n_{836}$ &(3,-2,1,2,-3)$^2$, 2, (2,1,-2)$^2$, 3, (3,2,1,-2,-3) &
\\ \hline 
$12n_{849}$  &(2,1,-2), 3, (3,2,1,-2,-3),1,(3,-2,1,2,-3), (-2,1,2), 2 &
\\ \hline 
$12n_{863}$ &(-2,1,2), 3, (3,-2,1,2,-3), (3,2,-3), (3,2,1,-2,-3)$^2$, (2,1,-2) &
\\ \hline 
$12n_{867}$ & (2,1,-2),(2,1,-2),1,(1,2,-1),3,(1,2,-1),2,3,3 &
\\ \hline 
$12n_{882}$ & (2,1,-2)$^3$, 1, (-2,1,2)$^2$, 2, 2 & QP-known
\\ \hline
\end{tabular}

\subsection{Tables of quasipositive knots with $\delta_3 >1$ }\label{subsec:delta3>1}

The following is the list of knots up to 12 crossings with $\delta_3 >1$ and $\delta_4=0$. We confirmed that they are all quasipositive. 
Therefore the following list, together with the previous list in Section \ref{subsec-tableQP}, gives a complete list of quasipositive knots up to 12 crossings, with possibly two exceptions $12_{n239}$ and $12_{n512}$ --- for $12_{n239}$ and $12_{n512}$ $\delta_{4}$ is only known to be $0$ or $1$, and we could not find quasipositive representatives of them. 

\hspace{-1cm}\begin{tabular}{|l||l|l|l|}
\hline
Knot & $\delta_3$ & QP braid notation & Comment
\\ \hline
$10_{155}$ & 3 &  (1,2,-1), (-1,-2,3,2,-3,2,1),(2,3,-2) & QP-known, mirror 
\\ \hline 
$10_{159}$ & 2 & (-2,-1,-1,2,1,1,2),(-2,1,2),(2,1,-2)$^2$ & QP-known
\\ \hline 
$11_{n22}$ & 2 & (-2,1,2)$^2$, (-1,2,1), (3,-2,3,2,-3), (3,2,-3)  &
\\ \hline
$11_{n40}$ & 2 &(2,1,-2),1,(-2,-2,1,2,2),(-2,-2,3,2,2),(2,3,-2)&
\\ \hline 
$11_{n46}$ & 2 & 1,(2,2,-1,2,1,-2,-2),(2,2,3,-2,-2),3,(3,2,-3)&
\\ \hline 
$11_{n50}$ & 2 &  1,(2,2,-1,-2,3,2,-3,2,1,-2,-2),(2,3,-2) & mirror
\\ \hline 
$11_{n54}$ & 2 &  (1,1,2,-1,-1),(2,1,-2),(2,3,-2),3,3 &
\\ \hline 
$11_{n71}$ & 2 &  1,(2,2,-1,2,1,-2,-2),(2,2,3,-2,-2),(3,2,-3),(3,2,-3) &
\\ \hline 
$11_{n75}$ & 2 &  1,1,(-2,-2,-2,1,2,2,2),(-2,-2,-2,3,2,2,2),3 & mirror
\\ \hline 
$11_{n87}$ & 2 & (-2,1,2)$^2$,(2,1,-2,3,2,-1,-2),(2,1,-2,1,2,-1,-2),3 & mirror
\\ \hline 
$11_{n127}$ & 2 &  (1,1,2,-1,-1),2,(2,1,-2),(2,3,-2),3  & mirror
\\ \hline 
$11_{n132}$ & 2 & (2,1,-2),(-2,3,-2,-3,3,3,2,-3,2),(-2,3,-2,-3,1,3,2,-3,2) & mirror
\\ \hline 
$11_{n146}$ & 2 &  1,2,(2,-1,-2,2,2,3,-2,-2,2,1,-2),3,(3,2,-3) &
\\ \hline 
$11_{n159}$ & 2 &  1,(2,2,-1,-2,2,2,1,-2,-2),(2,2,3,-2,-2),(2,1,-2),3 & mirror
\\ \hline 
$11_{n172}$ & 2 &(-2,1,2),(3,2,1,-2,-3,-1,-1,2,1,1,3,2,-1,-2,-3),3  & mirror
\\ \hline 
$11_{n176}$ & 2 & 1,1,(2,-1,-2,-2,1,2,2,1,-2),(2,-1,-2,-2,3,2,2,1,-2),3  & mirror
\\ \hline 
$11_{n178}$ & 2 & 1,(2,-1,-2,-2,3,2,2,1,-2),(2,-1,-2,-2,1,2,2,1,-2), & \\
& & (2,-1,-2,-2,3,2,1,-2,-3,2,2,1,-2),3 &
\\ \hline 
$11_{n184}$ & 2 & (2,1,-2),1,(3,-2,-3,3,3,2,-3),(3,-2,-3,-2,1,2,3,2,-3),(-3,2,3) &
\\ \hline 
$12_{n5}$ & 2  & (2,1,-2), 1, (-2,-3,4,3,-4,3,2), 4,4,(2,3,-2) & mirror
\\ \hline 
$12_{n80}$ & 2  & 1,(2,2,2,1,-2,-2,-2),(-3,2,1,-2,3)$^2$,3 & mirror
\\ \hline
$12_{n113}$ & 2  &   (2,1,-2),(1,1,2,-1,-1),2,2,2,2 & QP-known, mirror
\\ \hline
$12_{n116}$ & 2 & 1,(-3,-2,1,2,3),2,3,(3,2,-3) & mirror
\\ \hline
$12_{n140}$ & 2   &
(3,3,2,-3,-3)$^2$,(3,3,4,-3,-3),2,(3,2,1,-2,-3,4,3,2,-1,-2,-3), & \\
& &(3,2,1,-2,-3,1,3,2,-1,-2,-3) &
\\ \hline
$12_{n145}$ & 2  & 
(3,2,-3),4,(2,1,-2)$^2$,(-1,-2,-3,2,1,-3,4,3,-1,-2,3,2,1),& \\
& & (-1,-2,-3,2,1,-3,2,3,-1,-2,3,2,1)  & mirror 
\\ \hline 
$12_{n157}$ & 2   & 1,(-2,3,2),(-2,1,2),4,(2,-3,4,3,-2),(2,-3,2,3,-2) &
\\ \hline
$12_{n159}$ & 2 &
(-2,-2,1,2,2),(-3,2,3)$^2$,3,(2,1,-2) & mirror
\\ \hline
$12_{n190}$ & 2 & 
 1,1,(1,2,-1),2,2,(2,-1,-2,3,2,1,-2)  &QP-known, mirror
\\ \hline
$12_{n193}$ & 2 & 1,(1,1,2,-1,-1),3,(3,2,-3),(2,1,-2) & mirror
\\ \hline
$12_{n208}$ & 2& (3,2,-3),(2,1,-2),(-1,-1,-1,2,1,-2,1,1,1),3,&\\
& & (-1,-1,-1,3,2,-3,1,1,1)  & mirror
\\ \hline
$12_{n212}$ & 2 & (1,2,-1),(1,2,1,-2,-1),(1,3,-2,1,2,-3,-1),(3,2,-3,2,3,-2,-3),& \\
&  &(3,2,-3,1,3,-2,-3) & mirror
\\ \hline
$12_{n233}$ & 2 &  (-1,2,1,-2,1),(2,-1,2,1,-2),(2,1,-2)$^4$ & QP-known, mirror
\\ \hline
$12_{n236}$ & 2 & (1,2,-1),3,(1,3,2,1,-2,-3,-1),2,(2,2,1,-2,-2) & mirror
\\ \hline
$12_{n253}$ & 2 & (3,2,1,-2,-3),(3,3,2,-3,-3),(-2,1,2),(-2,-2,3,2,2),3 & mirror
\\ \hline
$12_{n318}$ & 3 & (-2,1,2),(3,2,-1,-2,-3,-3,2,3,3,2,1,-2,-3),&\\
& &(3,2,-1,-2,-3,1,3,2,1,-2,-3) & mirror
\\ \hline
$12_{n344}$ & 2 & 1,1,1,2,(2,2,2,1,-2,-2,-2)$^2$ & QP-known, mirror
\\ \hline
$12_{n345}$ & 3 & (2,2,2,1,-2,-2,-2),1,(1,2,-1)$^2$ & QP-known 
\\ \hline
$12_{n347}$ & 2 & 1,(-2,1,2),3,(3,3,3,2,-3,-3,-3)$^2$ & mirror
\\ \hline
$12_{n372}$ & 2 & (1,1,2,-1,-1)$^2$,3,(3,2,-3),(3,2,1,-2,-3) &
\\ \hline
$12_{n379}$ & 2 & (-1,3,2,1,-2,-3,1),(3,2,-1,-2,-3,3,2,-3,3,2,1,-2,-3),& \\
& & (3,2,-1,-2,-3,2,3,2,1,-2,-3)$^2$,(3,2,-1,-2,-3,3,3,2,1,-2,-3) & mirror 
\\ \hline
$12_{n393}$ & 2 & (2,1-2,4,3,-4,3,4,-3,4,-3,-4,2,-1,-2),& \\
& & (2,1,-2,4,3,-4,3,3,2,1,-2,-3,-3,4,-3,-4,2,-1,-2), & mirror \\
& &(2,1,-2,4,3,-4,3,2,-3,4,-3,-4,2,-1,-2),2  &  
\\ \hline
\end{tabular}

\pagebreak

\begin{tabular}{|l || l |l|l|}
\hline
Knot & $\delta_3$ & QP braid notation & Comment
\\ \hline
$12_{n409}$ & 2 & 1,(-3,2,3),(3,2,1,-2,-3),(2,-1,-2),(3,2,1,-2,-3),(2,1,-2)$^2$  & mirror
\\ \hline
$12_{n451}$ & 2 & (2,1,-2),3,1,(3,2,-3),(-3,-3,2,3,3) & mirror
\\ \hline
$12_{n454}$ & 2 & 1,(3,-2,1,2,-3),(2,-3,2,3,-2),3,3  & mirror
\\ \hline
$12_{n466}$ & 2 & 2,(2,2,2,1,-2,-2,-2),1,1,1,2 &  QP-known, mirror
\\ \hline
$12_{n467}$ & 3 &  1,(-2,-2,-2,-2,1,2,2,2,2)$^2$,2 & QP-known
\\ \hline
$12n_{469}$ & 2 & (-2,-2,1,2,2),2,3,(3,2,1,-2,-3),(2,1,-2) & mirror 
\\ \hline 
$12n_{487}$ & 2 & (-1,-2,3,2,1)$^2$,(2,1,-2),(1,2,3,-2,-1),2 &
\\ \hline 
$12n_{514}$ & 2& (2,1,-2),(3,2,-3),(3,2,1,-2,-3)$^2$,& \\
& &(3,2,1,-2,-3,-2,1,2,3,2,-1,-2,-3) & mirror
\\ \hline 
$12n_{522}$ & 2& 1,(-3,2,3),(3,2,1,-2,-3,2,1,-2,3,2,-1,-2,-3),(2,1,-2)$^2$ & mirror
\\ \hline 
$12n_{543}$ & 2& 2,3,(3,2,-3),(3,2,1,-2,-3),(1,1,2,-1,-1) & mirror
\\ \hline 
$12n_{570}$ & 2 &  1,2,2,2,(2,2,2,1,-2,-2,-2)$^2$ & QP-known, mirror
\\ \hline 
$12n_{572}$ & 2& 3,(1,3,2,-3,-1),(1,2,-1)$^2$,(2,2,1,-2,-2)  & mirror
\\ \hline 
$12n_{577}$ & 2& 1,(-3,2,3),(3,2,1,-2,-3,2,1,-2,3,2,-1,-2,-3)$^3$ &
\\ \hline 
$12n_{582}$ & 2& 1,(-3,-2,1,2,3),4,(-2,4,3,2,-3,-4,2) &
\\ \hline 
$12n_{604}$ & 2& (2,2,1,-2,-2),1,1,(-2,1,2)$^2$,2  &  QP-known, mirror
\\ \hline 
$12n_{606}$ & 2& (2-1,2,1,-2),(-2,1,2),(-3,-2,1,2,3),(-3,2,3),3 & mirror
\\ \hline 
$12n_{621}$ & 2& (-3,2,3),(3,4,-3),(3,2,2,1,-2,-2,-3),(2,1,-2),4,&\\
& & (-3,2,1,-2,3)  & mirror
\\ \hline 
$12n_{666}$ & 2& 1,(-2,1,2),(-2,-2,1,2,2),(-2,-2,-2,1,2,2,2),2,2 &  QP-known 
\\ \hline 
$12n_{667}$ & 2& (2,1,-2),3,(1,3,2,-3,-1),(1,2,2,1,-2,-2,-1),(2,1,-2) &
\\ \hline
$12n_{683}$ & 2&   (2,2,1,-2,-2),(2,1,-2),1,1,(1,2,-1),2 & QP-known, mirror
\\ \hline 
$12n_{684}$ & 2& (2,2,1,-2,-2),(2,1,-2),1,1,1,(-2,1,2) & QP-known, mirror
\\ \hline 
$12n_{685}$ & 2& 1,(3,2,-3),(-3,2,1,-2,3)$^2$,(2,-1,-2),3,(3,2,1,-2,-3) & mirror
\\ \hline 
$12n_{707}$ & 2&   (2,1,-2)$^3$,(-2,1,2),(-2,-2,1,2,2),2 & QP-known, mirror 
\\ \hline 
$12n_{708}$ & 4& (2,2,2,1,-2,-2,-2),(1,-2,1,2,-1) & QP-known
\\ \hline 
$12n_{717}$ & 2& 1,(3,-2,-3,-2,1,2,3,2,-3),(3,-2,-3,2,3,2,-3),& \\
& &(3,-2,-3,3,3,2,-3),(3,2,1,-2,-3) & mirror 
\\ \hline 
$12n_{719}$ & 2&  1,(-2,1,2),(-3,2,3),(-3,-3,2,3,3),(-3,-3,-3,2,3,3,3) &
\\ \hline 
$12n_{721}$ & 4&  1,
(2,-1,-1,-1,-1,2,1,1,1,1,-2) 
 & QP-known, mirror
\\ \hline 
$12n_{730}$ & 2&(-2,1,2),(3,2,-1,-2,-3,3,3,2,1,-2,-3), & \\
& & (3,2,-1,-2,-3,2,3,2,1,-2,-3),(2,1,-2),(3,2,-3) &
\\ \hline 
$12n_{742}$ & 2 &(-1,3,2,1,-2,-3,1),(3,-2,-3,1,3,2,-3),2,(2,1,-2),3 & mirror
\\ \hline 
$12n_{747}$ & 2 & (-1,2,1)$^2$,1,(-2,-2,1,2,2)$^2$,2 & QP-known, mirror
\\ \hline 
$12n_{748}$ & 3 &  (2,1-2,1,2,-1,-2),1,(-2,-2,1,2,2)$^2$ & QP-known
\\ \hline 
$12n_{767}$ & 3 &  (-1,2,1),(-1,-1,-1,-1,2,1,1,1,1)$^2$,2 & QP-known, mirror
\\ \hline 
$12n_{768}$ & 3 & 1,(-3,-3,-2,1,2,3,3),(-2,-2,3,2,2) & mirror 
\\ \hline 
$12n_{769}$ & 2 & (-2,1,2),3,(3,2,-3),(3,2,-3,3,2,1,-2,-3,3,-2,-3),(2,1,-2) & mirror
\\ \hline 
$12n_{771}$ & 2 & (2,1,-2),3,(3,2,1,-2,-3),(1,3,-2,1,2,-3,-1),(1,2,-1) & mirror
\\ \hline 
$12n_{811}$ & 2 & (2,2,1,-2,-2),(3,-2,3,2,-3),(3,-2,1,2,-3),(2,3,-2)$^2$ & 
\\ \hline 
$12n_{822}$ & 3 &  (2,2,1,-2,-2),(1,1,2,-1,-1),2,2 & QP-known
\\ \hline 
$12n_{829}$ & 3 & (2,2,1,-2,-2)$^2$,(-1,-1,2,1,-2,1,1),2 & QP-known, mirror
\\ \hline 
$12n_{831}$ & 3 &  (1-2,-2,1,2,2,-1)$^2$,2,(2,1,-2)  & QP-known
\\ \hline 
$12n_{838}$ & 2 &  1,(2,-1,-2,3,2,-3,2,1,-2),(4,3,-4),(-4,3,2,-3,4) &
\\ \hline 
$12n_{861}$ & 2 & (2,2,1,-2,-2),3,(3,2,1,-2,-3),(3,-2,1,2,-3),(3,2,-3) &
\\ \hline 
$12n_{862}$ & 2 & 2,3,(3,2,-1,2,1,-2,-3),(-1,3,2,1,-2,-3,1),(-1,-1,2,1,1) &
\\ \hline 
$12n_{871}$ & 2 & (1,2,-1),2,(-3,2,3),1,(-3,4,3,3,2,-3,-3,-4,3),4 &
\\ \hline 
$12n_{887}$ & 2 &  (2,1,-2),1,(-2,1,2)$^2$,(2,-1,2,1,-2)$^2$ & QP-known
\\ \hline  
\end{tabular}

\section*{Acknowledgements}
The authors would like to thank 
Jeremy Van Horn-Morris for sharing the list of knots whose strongly quasipositivity was previously unknown, 
Camila Ramilez for making an Excel file, and 
Inanc Baykur, Daryl Cooper,  John Etnyre, Charlie Frohman, and Matthew Hedden for useful conversations. 
JH was partially supported by The University of Iowa's Ballard and Seashore Dissertation Fellowship. 
TI was partially supported by JSPS Grant-in-Aid for Young Scientists (B) 15K17540.
KK was partially supported by NSF grant DMS-1206770 and Simons Foundation Collaboration Grants for Mathematicians.

\end{document}